\documentclass[a4paper,10pt]{article}
\usepackage[utf8]{inputenc}
\usepackage{amsmath,amssymb,amsthm}
\usepackage[colorlinks,citecolor=blue]{hyperref}
\usepackage{cleveref}
\usepackage{enumitem}
\usepackage{mathrsfs}
\usepackage{graphicx}
\usepackage{tikz}
\usetikzlibrary{patterns, shapes, arrows}
\usetikzlibrary{trees}
\usetikzlibrary{cd, graphs.standard}

\newcommand{\ZZ}{\mathbb{Z}}
\newcommand{\NN}{\mathbb{N}}

\newcommand{\QQ}{\mathbb{Q}}
\newcommand{\RR}{\mathbb{R}}

\newcommand{\TA}{\mathbb{A}}
\newcommand{\TB}{\mathbb{B}}

\newcommand{\TD}{\mathbb{D}}
\newcommand{\TE}{\mathbb{E}}

\newcommand{\tw}{t_{\mathsf{w}}}

\newcommand{\bu}{\overline{u}}
\newcommand{\bx}{\overline{x}}

\newcommand{\mX}{\mathsf{X}}
\newcommand{\mY}{\mathsf{Y}}

\newtheorem{theorem}{Theorem}[section] 
\newtheorem{proposition}[theorem]{Proposition} 
\newtheorem{conjecture}[theorem]{Conjecture} 
 
\newtheorem{lemma}[theorem]{Lemma}

\newcommand{\po}{\mathscr{P}}

\newcommand*\pFq[4]{{}_{#1}F_{#2}\biggl(\genfrac..{0pt}{}{#3}{#4};1\biggr)}

\newcommand{\down}{\mathscr{D}}
\newcommand{\up}{\mathscr{U}}

\newcommand{\Zeta}{\mathsf{Zeta}}
\newcommand{\tM}{\mathbb{M}}
\newcommand{\haut}{\operatorname{ht}}
\newcommand{\trunc}{\mathsf{T}}
\newcommand{\lap}{\mathcal{L}}
\newcommand{\rt}{r}
\newcommand{\nz}{\operatorname{nz}}
\newcommand{\rev}{\operatorname{Rev}}
\newcommand{\NC}{\mathsf{NC}}
\newcommand{\posi}{\operatorname{posi}}
\newcommand{\nega}{\operatorname{nega}}

\definecolor{racines}{RGB}{197,225,165}

\title{On posets and polytopes attached to arbors}
\author{F. Chapoton}
\date{\today}

\begin{document}

\maketitle

\begin{verse}
\textit{À la mémoire des petits bilbis.}
\end{verse}

\begin{abstract}
  Starting from the data of an arbor, which is a rooted tree with
  vertices decorated by disjoint sets, we introduce a lattice polytope
  and a partial order on its lattice points. We give recursive
  formulas for various classical invariants of these polytopes and
  posets, using the tree structure. For linear arbors, we propose a
  conjecture exchanging the Ehrhart polynomial of the polytope with
  the Zeta polynomial of the poset for the reverse arbor. The general
  motivation comes from a transmutation operator acting on
  $M$-triangles, which should link the posets considered here with
  some kind of generalized noncrossing partitions and generalized
  associahedra. We give some evidence for this relationship in several
  cases, including notably some polytopes, namely halohedra and
  Hochschild polytopes.
\end{abstract}

\section*{Introduction}

The combinatorial data that will be our starting point is a rooted
tree whose vertices are decorated with a partition of a finite
set. Such objects will be called \textit{arbors}. This article
introduces two constructions attached to any arbor $t$. The first
construction is a lattice polytope $Q_t$ inside some $\NN^n$ and the
second one is the partial order $P_t$ defined by coordinate-wise
comparison on the set of lattice points in this polytope.

Before explaining at length the motivation for these constructions
coming from Coxeter groups, cluster theory and representation theory,
let us describe our results.

First, it is shown that the polytopes $Q_t$, initially defined by
inequalities, are lattice polytopes, which means that they have
integral vertices. Moreover, they are simple polytopes and have no
interior lattice point. Their vertices are described in a recursive
way, using a projection map $\pi$ to a product of smaller polytopes in
the same family. The polytopes are also identified as Minkowski sums
of simplices. Basic properties of the posets $P_t$ are also stated:
they are graded with a unique minimum and the interval between this
minimum and any element is a product of total orders. This property
allows to view their Hasse diagrams as skeletons of cubical complexes.

Our next set of results consists of recursive formulas to compute
several invariants of the polytopes $Q_t$ and posets $P_t$. This uses
the projection map $\pi$ and often requires a finer invariant that
includes information about the height, which is the sum of
coordinates. A recursive formula is obtained for the Ehrhart
polynomial of the polytope $Q_t$, which counts lattice points in the
dilates, cf.~\cite{barvinok}. One considers then the function that
measures the volume in the hyperplane sections according to height,
and gets recursive formulas for this function and for its Laplace
transform. Similarly, recursive formulas are given for the Zeta
polynomial, the $M$-triangle and the $f$-vector of the poset
$P_t$. The Zeta polynomial is a classical invariant that counts chains
in the poset, see~\cite[\S 3]{stanley1974}. The $M$-triangle is a
polynomial in two variables, defined in~\cite{chap_2004} for any graded poset,
as a weighted sum of all Möbius numbers, similar to the more classical
characteristic polynomial introduced in~\cite{rota1964}. The
$f$-vector is defined by counting faces of the underlying cubical
complex.

In \Cref{sec:EhrhartZeta}, two conjectures are proposed in the case of
linear arbors. The first one claims that reversing the arbor
upside-down should exchange the Ehrhart polynomial of the polytope
$Q_t$ and the Zeta polynomial of the poset $P_t$. Here both
constructions interact in a rather surprising way. The second
conjecture claims that the $h$-vector is unchanged by reversing a
linear arbor. Here the $h$-vector of an arbor $t$ is the generating
polynomial of elements of $P_t$ according to the number of non-zero
coordinates.

In \Cref{sec:transmutation}, we introduce \textit{transmutation} as a
birational involution acting on the $M$-triangle. Transmutation is
central in the original motivation of the article and appears in the
results in the next sections. Here the reader may want to read the
motivation section below first.

In the five next sections, we gather evidence concerning transmutation
for several specific families of arbors. Let us say that two posets
are \textit{transmuted partners} if they share the same Zeta
polynomial and their $M$-triangles are related by transmutation.

In \Cref{sec:unarybinary}, we propose conjectures about unary-binary
arbors. In this case, one can define from the arbor $t$ a gentle
quiver-with-relations whose lattice of support $\tau$-tilting objects
is finite. Applying the core-label-order construction should lead to
transmuted partners of the posets $P_t$. This is also stated in the
more combinatorial language of Stokes posets for quadrangulations~\cite{baryshnikov, stokes}.

In \Cref{sec:A}, we consider the case of arbors of type $\TA$, meaning
linear arbors with no multiple vertex. We introduce a more general
family of posets, including the posets $P_t$ for arbors of type $\TA$. We
compute the Zeta polynomials and the $M$-triangles of these more general
posets. This implies that noncrossing partitions of Coxeter type $\TA$ are
transmuted partners of the posets $P_t$ for arbors of type $\TA$.

In \Cref{sec:B}, we consider the case of arbors of type $\TB$, meaning
arbors with only one vertex. We introduce a more general family of
posets, including the poset $P_t$ for arbors of type $\TB$. We compute
the Zeta polynomials and the $M$-triangles of these more general
posets. This implies that noncrossing partitions of Coxeter type $\TB$
are transmuted partners of the posets $P_t$ for arbors of type $\TB$.

In \Cref{sec:halohedra}, we turn to arbors with one simple vertex and
one multiple vertex. The aim is to relate them to the family of simple
polytopes named halohedra~\cite{devadoss1, devadoss2}. We prove that
the $h$-vector of the polytope coincides with the $h$-vector of the
arbor.

In \Cref{sec:hochschild}, we turn to arbors that are corollas, namely
with no multiple vertex and such that all non-root vertices are
attached to the root vertex. The aim is to relate these arbors to the
family of Hochschild polytopes~\cite{saneblidze, san_rivera} and the
associated Hochschild lattices~\cite{combe_hoch}. We conjecture that
the core-label-orders of Hochschild lattices are transmuted partners
of the posets $P_t$. We prove that the $h$-vector of the Hochschild
polytope coincides with the $h$-vector of the arbor.

\subsection{Motivation}\label{subsec:motivation}

Let us now try to explain, with the help of the following figure,
where this work comes from.

\tikzstyle{block} = [rectangle, draw, fill=blue!10, 
    text width=9em, text centered, minimum height=1cm, rounded corners]
\tikzstyle{block2} = [rectangle, draw, fill=orange!10, 
    text width=8em, text centered, minimum height=1cm, rounded corners]
\tikzstyle{block3} = [rectangle, draw, fill=green!10, 
    text width=8em, text centered, minimum height=1cm, rounded corners]
\tikzstyle{line} = [draw, -latex']

\begin{center}
\begin{tikzpicture}[node distance =4cm, auto]
    \node at (0,2) [block] (polytope) {simple polytopes (generalized associahedra)};
    \node at (-2,0) [block] (tamari) {congruence-uniform lattices\\ (Cambrian lattices)};
    \node at (0,-2) [block3] (noncrossing) {graded lattices (noncrossing\\ partitions)};
    \node at (2,0) [block2] (posetPt) {graded posets of arbors};
    \node at (6,0) [block2] (polytopeQt) {lattice polytopes\\ of arbors};
    \path [draw, dashed] (posetPt) -- (polytope);
    \path [draw, dashed] (posetPt) -- (noncrossing);
    \path [line] (tamari) -- (noncrossing);
    \path [draw, thick, double] (tamari) -- (polytope);
    \path [draw, thick, double] (posetPt) -- (polytopeQt);
\end{tikzpicture}
\end{center}

Let us start with the leftmost box. The theory of Cambrian lattices~\cite{reading_cambrian} defines, for any finite Coxeter group $W$ and
Coxeter element $c \in W$, a finite lattice. For crystallographic
simply-laced Coxeter groups, these lattices can also be described
inside the theory of cluster algebras, using green mutations, see~\cite{keller_green} for a survey. There is an alternative
representation-theoretic viewpoint using inclusion of torsion classes
for finite-dimensional algebras, considered also using support
$\tau$-tilting objects~\cite{tau_tilting}. It is known that finite
lattices of torsion-classes are congruence-uniform~\cite{garver_lattice, DIRRT}. It is also known that Cambrian lattices
are congruence-uniform, as quotients of the weak order on the Coxeter
group~\cite{reading_cambrian}.




Let us move to the top box, closely tied to the previous one. It is
known that the Hasse diagrams of Cambrian lattices are skeletons of
some simple polytopes, sometimes called generalized associahedra
\cite{CFZ}. In this setting,
one can define the $F$-triangle, a polynomial in two variables that
refines the $f$-vector of the polytope \cite{chap_2004} and depends on
the choice of a distinguished vertex. Its definition, as recalled in
\Cref{app:triangles}, is applied to the polar polytope seen as a
spherical simplicial complex with a distinguished facet.

Consider now the bottom box. Starting from a congruence-uniform
lattice in the leftmost box, one can define another lattice, using the
canonical join representation. This construction was introduced by
Reading in~\cite[\S 9.7]{reading_regions} under the name of
\textit{shard-intersection-order}. This is also known as the
\textit{core-label-order}, in the terminology
of~\cite{muhle_core}. For Cambrian lattices, the core-label-order only
depends on $W$ and recovers the $W$-noncrossing-partitions lattice of
Brady-Watt~\cite{brady_watt} and Bessis~\cite{bessis}. As these
lattices are graded, one can define their $M$-triangles as recalled in
\Cref{app:triangles}.



In the Cambrian context, the $F$-triangle of the generalized
associahedra is related to the $M$-triangle of the lattice of
$W$-noncrossing-partitions by the explicit birational
transformations~\eqref{M_from_F} and~\eqref{F_from_M} recalled in
\Cref{app:triangles}. This equality was conjectured
in~\cite{chap_2004} and proved in~\cite{athanasiadis}. These
birational transformations are expected to be more widely applicable
in similar contexts, see for example~\cite{garver_nonkissing}. Note
that this birational relationship implies that $h$-vectors of
generalized associahedra also enumerate non-crossing partitions
according to their rank.

\smallskip

All this forms a well-traveled landscape, involving congruence-uniform
lattices and directed skeletons of polytopes. Our main claim is that
there should exist another aspect of this story, related to the
previous landscape by transmutation.

The \textit{core idea} is that some posets from the bottom box (noncrossing
partitions and similar posets) should be \textit{transmuted partners}
of some posets from the middle box (associated with arbors). By this
expression, we mean that two posets share the same Zeta polynomial and
that their $M$-triangles are related by an explicit birational
involution that we choose to call \textit{transmutation}.

The starting point for the project was in fact in the other direction:
starting from the classical lattices of noncrossing partitions, we
have found their transmuted partners. This has led us to the present
framework of arbors.

We will prove in \Cref{sec:A} that for linear arbors in which every
vertex has only one element, the transmuted $M$-triangle is indeed the
$M$-triangle of the noncrossing-partitions lattice of type
$\TA$. Moreover, we prove that they share the same Zeta
polynomial, so they are indeed transmuted partners as claimed.

In a similar vein, it is shown in \Cref{sec:B} that the noncrossing
partitions lattices of type $\TB$ are transmuted partners of the posets
$P_t$ for arbors with one vertex.

\begin{center}
  $\clubsuit$~$\clubsuit$~$\clubsuit$
\end{center}

Having introduced a general theory on the arbor side, it is natural to
ask if one can use transmutation in the other direction. 

The best hope would be the following statement for every arbor $t$:

There should exist a simple polytope $\Phi_t$, some way to orient its
skeleton to find the Hasse diagram of a congruence-uniform
``Cambrian-like'' lattice $\Lambda_t$ and some ``noncrossing-like''
lattice $\NC_t$ built from $\Lambda_t$ by the ``core-label-order''
construction. The lattice $\NC_t$ should be a transmuted partner of
$P_t$.  If the arbor $t$ has size $n$ and the poset $P_t$ has $N$
elements, then $\Phi_t$ should be a simple polytope of dimension $n$
with $N$ vertices.

There is so far no idea on how to achieve that in general. A proposal
for unary-binary arbors is made in \Cref{sec:unarybinary} in terms of gentle algebras.

\smallskip

A much weaker statement would be to produce starting from $t$ the
simple polytope $\Phi_t$ together with a distinguished vertex (as in
the top box) whose $F$-triangle can be transformed by formulas in
\Cref{app:triangles} into an $M$-triangle (as in the bottom box) that
would be the transmutation of the $M$-triangle of $P_t$. In
particular, \Cref{transmuted_h} says that the $h$-vector of $t$ is
then the diagonal of the $M$-triangle of the lattice $\NC_t$, hence
describes the grading of this poset and also the $h$-vector of the
polytope $\Phi_t$.

Again, there is no known idea on how to do that. Two proposals are
made in \Cref{sec:halohedra} and \Cref{sec:hochschild} that should
relate in this way some arbors with known families of polytopes. The
evidence is given by a coincidence of $h$-vectors, which would follow
from the expected transmutation of $M$-triangles.

Assuming that we can produce the polytope $\Phi_t$ endowed with its
distinguished vertex, one could then consider the simplicial complex
defined as the dual of the simple polytope $\Phi_t$. One expects that
this simplicial complex is flag, and moreover that its Gamma-triangle
with respect to the distinguished facet, as defined in
\cite{chap_gamma}, is nonnegative. This holds for all arbors of size at
most $13$.

\smallskip

Even weaker is the following conjecture, keeping only minimal
information.

\begin{conjecture}\label{very_weak}
  For every arbor $t$, there exists a simple polytope $\Phi_t$ such
  that the $h$-vector of $t$ is equal to the $h$-vector of $\Phi_t$.
\end{conjecture}


\subsection{Further remarks}

The posets that we attach to arbors of type $\TA$ are special cases of
the \textit{avalanche posets} introduced by Combe and Giraudo in
\cite{combe_giraudo2, combe_giraudo1}. Despite their very simple
nature and the fact that their cardinality is a Catalan number, they
do not seem to have been studied much.

The polytopes that we attach to arbors of type $\TA$ have Catalan-many
lattice points. They in fact coincide with a special case of the
\textit{Pitman-Stanley polytopes} introduced in \cite{pitman_stanley}, when all
parameters in the definition are set to $1$.

Our notion of pairs of posets being transmuted partners seems to have
some relationship, still to be clarified, with the notion of
\textit{Whitney duality} introduced in \cite{gonzalez_hallam}. Whitney
duality relates pairs of posets that exchange their Whitney
polynomials of the first and second kind. Some examples involve
noncrossing partition lattices. One can in particular note the
intriguing fact that the subposet of the poset
$\operatorname{NCDyck}_4$ of \cite[Fig. 9]{gonzalez_hallam} defined by
restriction to indices in the order $1,2,3,4$ is isomorphic to the
poset $P_t$ for the arbor of type $\TA$ on $3$ elements.

Note that our notion of \textit{arbor} has been considered in the
literature under various names, including ``multilabelled trees'' and ``rooted trees
of non-empty sets'' \cite{murua, berland}.

In our examples of \textit{transmuted partners}, we observe more
properties, for example the leading coefficients of the $q$-Zeta
polynomials~\cite{qZeta} are reciprocal of each other as polynomials in $q$.

\smallskip

Acknowledgments: Thanks to Pierre Baumann, Alin Bostan, Marc
Mezzarobba, Philippe Nadeau and Vincent Pilaud for illuminating
discussions and constructive suggestions. Thanks to the referee for a
careful report and very pertinent information.

\section{Notations and basic properties}\label{sec:notation}

Let us start by describing our main combinatorial objects, using the
language of species. For an introduction to the theory of species, see \cite{BLL}.

Let $S$ be the species of sets and $S_{\geq 1}$ be the species of
non-empty sets. Let us consider the species $A$ defined by the
implicit equation
\begin{equation*}
  A = S_{\geq 1} \times S(A).
\end{equation*}
In words, this means that the $A$-structures on a finite set $I$ are
rooted trees, whose vertices are labeled by pairwise disjoint
non-empty subsets of $I$, and the full set of labels is $I$.

Such a rooted tree will be called an \textit{arbor} on the set $I$. For example, here is an arbor on the set $\{a,b,c,d,e,f\}$:
\begin{center}
\begin{tikzpicture}
[grow=down]
\tikzstyle{every node}=[draw,shape=circle]

\node [fill=racines,thick] {$a,f$}
  child {node {$b,e$}}
  child {node {$d$}}
  child {node {$c$}}
;
\end{tikzpicture}.
\end{center}
This example will be used as a running example.

Arbors will either be depicted with their root vertex on the top and
growing down as above, or with their root vertex on the left and
growing right. The cardinality of a vertex $v$ will be denoted
$|v|$. If $|v| \geq 2$, then the vertex $v$ is \textit{multiple}. The
root vertex of an arbor $t$ will be denoted $\rt_t$.

For a vertex $v$ of an arbor $t$, any vertex $w$ such that the unique
path in $t$ from $w$ to the root vertex $\rt_t$ passes through $v$ is
called a \textit{descendant} of $v$. This condition is denoted by
$v \preceq w$. The set of all descendants of $v$ form the sub-tree
rooted at $v$.

To every vertex $v$ of an arbor, one associates the union of all
labels that appear in some descendant of $v$. This set is denoted
$\down(v)$, standing for ``down''. In the running example, the sets $\down(v)$
are $\{a,b,c,d,e,f\}$, $\{b,e\}$, $\{d\}$ and $\{c\}$.

When $|I| = n$, the arbor $t$ has \textit{size} $n$.  


\smallskip

To every arbor $t$ on a set $I$, we will associate both a partial
order and a polytope, closely related to each other. Let us start with
the polytope.

The ambient space is the vector space over $\RR$ with coordinates
${(x_i)}_{i\in I}$. The polytope $Q_t$ is defined by the inequalities
$0 \leq x_i$ for every $i \in I$ and 
\begin{equation}
  \label{def_inegalite}
  \sum_{i \in \down(v)} x_i \leq |\down(v)|
\end{equation}
for every vertex $v$. In particular, the inequality attached to the
root vertex $\rt_t$ says that the sum of all coordinates is less than
or equal to $|I|$, so that this indeed defines a bounded polyhedron.

In the running example, besides all coordinates
being nonnegative, the defining inequalities are:
\begin{equation*}
  x_c \leq 1,\quad x_d \leq 1,\quad x_b+x_e \leq 2,\quad x_a+x_b+x_c+x_d+x_e+x_f \leq 6.
\end{equation*}

The following property is clear from the definition.
\begin{lemma}
  Points with all coordinates in the interval $[0,1]$ are contained in $Q_t$. In
  other words, the polytope $Q_t$ contains the standard hypercube.
\end{lemma}

The definition of polytopes $Q_t$ by inequalities implies clearly that
they are anti-blocking polytopes, in the sense of Fulkerson~\cite{fulkerson}, see also~\cite[\S 9.3]{schrijver}.

\smallskip
We will often use the height function $\haut$ which is the sum of
coordinates. Note that the height function is bounded above by $|I|$,
which is attained at the point $(1,1,\ldots,1)$.

Let now $t$ be an arbor on a set of cardinality $n \geq 1$. Let
$(t_1,\ldots,\tw)$ be the sub-trees of the root vertex of $t$. Let $r$
be the size of the root vertex of $t$. Then the projection $\pi$ along
the coordinates in the root vertex to the other coordinates defines a
map from $Q_t$ to the product $\prod_k Q_{t_k}$. This projection $\pi$
is surjective as the fiber over any point of total height $h$ in
$\prod_k Q_{t_k}$ is the intersection of $\RR_{\geq 0}^r$ with the
half-space $\sum_{i \in \rt_t} x_i \leq n - h$. Note that the total
height on $\prod_k Q_{t_k}$ is less than $n - r$. In particular, the
fibers are simplices of strictly positive dimension.

\begin{lemma}
  The polytope $Q_t$ is a simple lattice polytope. 
\end{lemma}
\begin{proof}
  This is true for arbors of size $1$, where $Q_t$ is the segment $[0,1]$.

  Assume that this is known for all arbors with size strictly less
  than $n$. Let $\pi(Q_t)$ denote $\prod_k Q_{t_k}$.

  Let $z$ be a vertex of $Q_t$. The point $z$ must be an extremal
  point in the fiber $\pi^{-1}(\pi(z))$. So its coordinates with
  indices in the root vertex $\rt_t$ are either all $0$ or exactly one
  of them is not $0$ and equal to $n-h$, where $h$ is the
  height of $\pi(z)$. If these coordinates are all $0$, then $\pi(z)$
  must be a vertex of $\prod_k Q_{t_k}$. In this case, we obtain a
  simple vertex with integer coordinates by induction hypothesis.

  In the case of exactly one non-zero coordinate, write
  $z = (\pi(z),n-h,0,\ldots,0)$ in the coordinate system where indices
  in the root come last. Let us show that the point $\pi(z)$ must
  still be a vertex of $\pi(Q_t)$. Assume by contradiction that
  $\pi(z)$ is not extremal. Then there exist distinct $\bar{z}_1$ and
  $\bar{z}_2$ in $\pi(Q_t)$ whose middle is $\pi(z)$. Let $h_1$, $h_2$
  be the heights of $\bar{z}_1$ and $\bar{z}_2$. Then the height $h$
  of $\pi(z)$ is $(h_1+h_2)/2$. Let $z_1=(\bar{z}_1,n-h_1,0,\ldots,0)$
  and $z_2=(\bar{z}_2,n-h_2,0,\ldots,0)$ in the same coordinate
  system. Then both $z_1$ and $z_2$ are in $Q_t$, distinct, and their
  middle is $z$. Therefore $z$ is not extremal, which is absurd. Hence
  $\pi(z)$ is a vertex of $\pi(Q_t)$ and it has integer coordinates
  and integer height $h$ by induction. Therefore $z$ also has integer
  coordinates.

  It follows that there are $r+1$ vertices of $Q_t$ in the fiber over
  each vertex of $\pi(Q_t)$. They all have integer coordinates. From
  the description above, one can also deduce that they are simple.
\end{proof}

One can extract from the proof above a recursive description of the
vertices of $Q_t$ themselves. Let us call \textit{defect} of a point
$z$ in $Q_t$ at a vertex $v$ the difference
\begin{equation}
  \delta_v(z) = |\down(v)| - \sum_{i \in \down(v)} x_i(z) \geq 0.
\end{equation}

\begin{lemma}\label{Qt_vertices}
  For each vertex $v$ of the arbor $t$ and for every vertex $z$ of the
  polytope $Q_t$,
  \begin{itemize}
    \item either all coordinates $x_i(z)$ with $i \in v$ vanish,
    \item or exactly one coordinate $x_i(z)$ with $i \in v$ does not vanish.
  \end{itemize}
  At every vertex $v$ where all coordinates vanish, the defect
  $\delta_v(z)$ is at least $|v|$. 
  At every vertex $v$ where one coordinate is non-zero, the defect
  $\delta_v(z)$ must vanish. 
\end{lemma}

This implies that, assuming all coordinates of a vertex $z$ of $Q_t$ for descendant
vertices of a vertex $v$ of $t$ are fixed, the unique possible
non-zero value in the vertex $v$ is therefore determined by the
zero-defect condition, which means that the defining inequality
\eqref{def_inegalite} associated with $v$ must be an equality.

\begin{lemma}
  The polytope $Q_t$ has no interior lattice point.
\end{lemma}
\begin{proof}
  Every interior lattice point must have all coordinates $x_i >
  0$, but the sum of coordinates must be an integer strictly less than
  $|I|$.
\end{proof}

\smallskip

Let us now define the poset $P_t$. The elements of $P_t$ are the
lattice points of $Q_t$. The partial order is defined by comparison of
coordinates, namely $x \leq y$ if and only if $x_i \leq y_i$ for all
$i \in I$. The poset $P_t$ has a unique minimum $0$ and is graded by
the height function $\haut$.

From the definition of $P_t$, one deduces the following property.
\begin{lemma}
  For every element $a$ in $P_t$, the interval $[0,a]$ is a product of
  total orders whose lengths are the non-zero coordinates of $a$ plus one.
\end{lemma}
This property implies that one can consider the Hasse diagram as the
skeleton of a cubical complex, by glueing together the cubical
complexes for products of total orders.

In the running example, the polytope $Q_t$ is a $6$-dimensional simple
lattice polytope with $f$-vector $(36, 108, 141, 102, 43, 10, 1)$. The
poset $P_t$ has $330$ elements, which are the lattice points of
$Q_t$. The cubical complex of $P_t$ has $f$-vector
$(330, 990, 1152, 654, 186, 24, 1)$. A simpler example is illustrated in \Cref{fig:1}.


\begin{figure}[ht]
\begin{center}
\includegraphics[height=1cm]{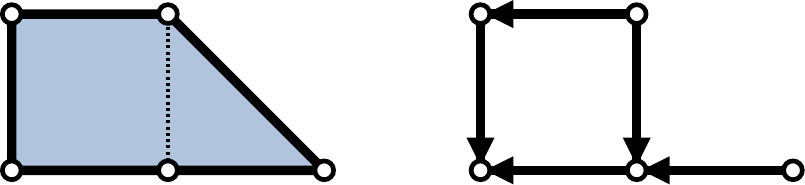}
\end{center}
  \caption{\label{fig:1}Polytope $Q_t$ and poset $P_t$ for the arbor $\colorbox{racines}{\{$x$\}} - \{y\}$.}
  
\end{figure}


\smallskip

Note that both the polytope $Q_t$ and the poset $P_t$ are species, in
the sense that every bijection $I \simeq J$ that induces a relabeling
of $t$ also induces an isomorphism of polytopes and posets. We will
from now on work with isomorphism classes of arbors.

\subsection{Minkowski decomposition}\label{subsec:minkowski}

Let $t$ be an arbor on the set $I$. Let us show that the polytope $Q_t$ admits a nice
decomposition as a Minkowski sum of simplices.

Let us first describe the Minkowski summands. There will be a summand
$U_v$ associated with every vertex $v$ of $t$.

For every vertex $v$, define a subset $\up(v)$ of $I$, standing for
``up'', as the union of the sets attached to all vertices of $t$ in
the unique path in $t$ from the root $\rt_t$ of $t$ to the vertex
$v$. The polytope $U_v$ is the convex hull inside $\RR^{\up(v)}$ of
the origin and all $|v|e_i$ where $e_i$ are the basis vectors of
$\RR^{\up(v)}$. Note that it is a Minkowski multiple of the basic
simplex whose vertices are the origin and the $e_i$.

In the running example, the subsets $\up(v)$ are
$\{a,f\},\{a,f,b,e\},\{a,f,d\}$ and $\{a,f,c\}$.

Let
\begin{equation}
  \tM_t = \sum_{v} U_v
\end{equation}
be the Minkowski sum of these simplices. We will show that
$\tM_t = Q_t$ by inclusion in both directions.

\smallskip
\begin{lemma}
  Every point in $\tM_t$ satisfies all the
  defining inequalities of $Q_t$.
\end{lemma}
\begin{proof}
  This is clear for the inequalities $x_i \geq 0$ for $i \in I$,
  because all Minkowski summands $U_v$ have nonnegative coordinates. So let
  us consider the inequality attached to a vertex $v$ of $t$. Fix an
  element $e$ in the Minkowski sum $\tM_t$, written as
  \begin{equation*}
    e = \sum_{w} e_w,
  \end{equation*}
  where $e_w \in U_w$. The inequality to check is
  \begin{equation*}
    \sum_{i \in \down(v)} x_i(e) \leq |\down(v)|.
  \end{equation*}
  So we need to bound the sum
  \begin{equation*}
    \sum_{i \in \down(v)} \sum_w x_i(e_w).
  \end{equation*}
  When $w$ is not a descendant of $v$, the Minkowski summand $U_w$
  involves only coordinates in $\up(w)$ which is disjoint from
  $\down(v)$. Therefore the sum above can be restricted to $w$ in the set of descendants of $v$:
  \begin{equation*}
    \sum_{v \preceq w} \sum_{i \in \down(v) \cap \up(w)} x_i(e_w).
  \end{equation*}
  By definition, the summand $U_w$ is a dilated basic simplex, hence
  \begin{equation*}
    \sum_{i \in \up(w)} x_i(e_w) \leq |w|.
  \end{equation*}
  Therefore our full sum is bounded above by
  \begin{equation*}
    \sum_{v \preceq w} |w| = |\down(v)|.
  \end{equation*}
\end{proof}

\begin{lemma}
  Let $t$ be an arbor on the set $I$. Every vertex of the polytope $Q_t$ belongs to $\tM_t$.
\end{lemma}

The proof is done by induction on the size of the arbor. This clearly
holds for arbors with just one vertex. The proof will use the
recursive description of vertices given in \Cref{Qt_vertices} and the defects.

\begin{proof}
  Let $(t_1,\ldots,\tw)$ be the sub-trees of the root vertex $\rt_t$
  of $t$. Let $z$ be a vertex of $Q_t$.  By
  induction on the size of $t$, the statement holds for all sub-trees
  $t_k$. As the image of $\pi$ is the direct product of the polytopes
  $Q_{t_k}$, one can take the direct product of the Minkowski
  decompositions of $Q_{t_k}$. Therefore the image $\pi(z)$ can be
  written as a sum
  \begin{equation}
    \label{pi_sum}
    \pi(z) = \sum_{w \not= \rt_t} e_w,
  \end{equation}
  where $e_w$ is in the Minkowski summand $U_w$ for some vertex $w$ in some $t_k$.

  If all coordinates $x_i(z)$ for $i \in \rt_t$ vanish, then to each $e_w$
  in the sum \eqref{pi_sum} one can associate a point $f_w$ with the
  same coordinates plus additional zero coordinates for all vertices
  $i \in \rt_t$. This defines a point $f_w$ in the summand $U_w$ of
  $\tM_t$, as the sum of coordinates in $\up(w)$ is unchanged. The sum
  $\sum_{w \not= \rt_t} f_w$ is therefore in $\tM_t$ and coincides with
  $z$ by construction. 

  \smallskip

  Assume now that exactly one coordinate $x_i(z)$ with $i \in \rt_t$
  is non-zero. Because the defect $\delta_{\rt_t}(z)$ vanishes, one
  gets
  \begin{equation*}
    x_i(z) + \sum_{j \not \in \rt_t} x_j(z) = n = |\rt_t| + \sum_k \sum_{w \in t_k} |w|.
  \end{equation*}
  It follows that
  \begin{equation}
    \label{initiale}
    x_i(z) - |\rt_t| = \sum_k \delta_k(z),
  \end{equation}
  where $\delta_k(z)$ is the defect of $z$ with respect to the root vertex of $t_k$.




  Consider the set $S$ of vertices that can be reached from the root
  $\rt_t$ by repeatedly moving to the root $w$ of a sub-tree with
  non-zero defect $\delta_w(z)$. Let $w$ be a vertex in $S$. Then from
  the element $e_w$ in the summand $U_w$ of some $\tM_{t_k}$, define
  $f_w$ by adding a new coordinate $x_i$ with value $|w|$ plus other
  zero coordinates for other vertices $j \in \rt_t$.

  Let us check that $f_w$ is an element in the summand $U_w$ of
  $\tM_t$. Certainly all coordinates of $f_w$ are nonnegative. Because
  the defect $\delta_v(z)$ is not zero for every vertex
  $v \not = \rt_t$ on the path between $\rt_t$ and $w$, all
  coordinates $x_j(z)$ with $j$ in $\up(w) \setminus \rt_t$ are
  zero. The sum of all coordinates of $f_w$ is therefore $|w|$. It
  follows that $f_w$ is in the summand $U_w$ of $\tM_t$.

  Let us consider $z' = \sum_{w \not= \rt_t} f_w$. Then for all
  $j \not \in r$, the coordinates $x_j(z)$ and $x_j(z')$ coincide by
  construction. For $j\in \rt_t$ not equal to $i$, the coordinates
  $x_j(z)$ and $x_j(z')$ both vanish. Only the coordinate $x_i$ may
  differ between $z$ and $z'$. By construction $x_i(z')$ is the sum of
  $|w|$ over all vertices $w$ in the reachable set $S$.

  Let us prove that this sum over $S$ is equal to $x_i(z) - |\rt_t|$. Let
  us start with \eqref{initiale} or more precisely its restriction to
  sub-trees $t_k$ whose root has non-zero defect. For every such $k$,
  the defect $\delta_k(z)$ is the sum of the size of the root of $t_k$
  plus the sum of the defects of the roots of sub-trees of
  $t_k$. Iterating this process over the set $S$, one obtains the
  expected equality:
  \begin{equation*}
    x_i(z) - |\rt_t| = \sum_{w \in S} |w|.
  \end{equation*}

  Let then $f_{\rt_t}$ be the element of the summand $U_{\rt_t}$ with
  coordinate $i$ equal to $|\rt_t|$ and other coordinates $0$. Then
  $z' + f_{\rt_t} = z$, and therefore $z \in \tM_t$.
\end{proof}

\smallskip

One can use this Minkowski decomposition to define a canonical
polynomial whose Newton polytope is $Q_t$, namely
\begin{equation}
  \prod_{v} {\left(1 + \sum_{e \in \up(v)} x_e\right)}^{|v|}.
\end{equation}

In the case of the arbors of type $\TA$, namely linear arbors with no
multiple vertex, these polynomials have appeared in the literature in
relation with Catalan numbers. They are item 210 in Stanley's book
\cite{stanley_catalan}. According to the errata of this book, these
polynomials first appeared in 1982 in a problem by J. Shallit in
\cite{shallit82, shallit86}. They have been studied recently in
\cite{fried2021}.

\smallskip

Let us state a corollary of the Minkowski decomposition. Postnikov
introduced in~\cite{post09} the notion of $Y$-generalized
permutohedra, defined as Minkowski sums of scaled standard simplices.

Let $e_0,e_1,\ldots,e_d$ denote the standard basis of $\ZZ^{d+1}$. A
scaled standard simplex is the convex hull of a set of vectors
${\{k_i e_i\}}_{i \in I}$ where $I$ is a non-empty subset of
$\{0,1,\ldots,d\}$ and $k_i$ are positive integers.

\begin{lemma}\label{y_gen}
  The polytope $Q_t$ is, up to unimodular equivalence, a
  $Y$-generalized permutohedra.
\end{lemma}
\begin{proof}
  Let us consider the finest Minkowski decomposition of $Q_t$, where
  each summand $U_v$ is further decomposed as a Minkowski sum of equal
  terms. In this decomposition, each summand is the convex hull in
  $\ZZ^d$ of $0$ and some vectors ${\{e_i\}}_{i \in I}$ where $I$ is a
  non-empty subset of $\{1,\ldots,d\}$. This defines a polytope
  $\Delta_I$ in $\ZZ^d$, to which one associates a standard simplex
  $\Delta'_I$ in $\ZZ^{d+1}$ by replacing the vertex $0$ by the basis
  vector $e_0$. Let us consider the linear map $\rho$ from $\ZZ^{d+1}$
  to $\ZZ^d$ that forgets the first coordinate. Then $\rho$ maps
  $\Delta'_I$ to $\Delta_I$ and is moreover a unimodular equivalence
  from $\Delta'_I$ to $\Delta_I$.

  Let us write $Q_t = \sum_{s \in S} \Delta_{I_s}$ and
  $Q'_t = \sum_{s \in S} \Delta'_{I_s}$ over the appropriate indexing
  set $S$. Then $\rho$ induces an invertible integral affine map from
  the hyperplane $H$ in $\ZZ^{d+1}$ where coordinates sum to $|S|$ to
  $\ZZ^d$, that maps $Q'_t$ to $Q_t$. This is the expected unimodular
  equivalence.
\end{proof}

\section{Invariants}\label{sec:invariants}

We will be interested in several invariants of the polytope $Q_t$ and the
poset $P_t$. Let us explain how they can be computed by induction.

The basis of the inductions is always the unique isomorphism class of
arbor on a set with one element. The polytope $Q_t$ is in this case
the line segment $[0,1]$ in $\RR$ and the poset $P_t$ is a total order with
$2$ elements.

The induction steps will use the following property or some variation
of it, that come from the surjective projection map $\pi$ described in
\Cref{sec:notation}. Let $t$ be an arbor and $t_1,\ldots,\tw$ be the
sub-trees of the root vertex of $t$. If $e$ is a lattice point in
$Q_t$, then one can describe $e$ using one lattice point $e_k$ in
$Q_{t_k}$ for every $1 \leq k \leq \mathsf{w}$ by restriction, plus the
additional data of the coordinates of $e$ with indices in the root
vertex. The only constraint on these root coordinates is their sum,
namely the height of $e$ has to be the sum of the heights of all $e_k$
plus the sum of the root coordinates. The same statement holds for all
real points.

\subsection{Ehrhart polynomial}

Let us first consider the Ehrhart polynomial $E_t$ of the lattice
polytope $Q_t$. Recall that this is the unique polynomial $E_t$ in a
variable $u$ such that, for every integer $m \geq 1$, $E_t(m)$ is the
number of lattice points in the dilated polytope $m Q_t$. A standard
reference on the subject is \cite{barvinok}.

In order to compute the Ehrhart polynomial, we will need to refine the
count according to height. This is done in a somewhat indirect way,
maybe not optimal.

For every integer $m \geq 1$, let us introduce a polynomial $F_{t,m}$
in the variable $X$ that counts lattice points in the dilated polytope
$m Q_t$ with weight $X^h$ at height $h$. For an arbor of size $n$, the
height on $m Q_t$ is bounded above by $m n$, so $F_{t,m}$ has degree $m n$.

Note that this definition of $F_{t,m}$ could be extended \textit{verbatim} to
any lattice polytope in $\RR_{\geq 0}^n$. In this generality, it is clearly
multiplicative with respect to direct product of polytopes.

For the arbor $t$ on one element,
\begin{equation}
  F_{t,m} = 1 + X + \cdots + X^m.
\end{equation}

Let $t$ be an arbor on a set of cardinality $n \geq 2$. Let
$(t_1,\ldots,\tw)$ be the sub-trees of the root vertex of $t$. Let
$r$ be the size of the root vertex of $t$.

\begin{proposition}\label{calcul_ehr}
  Let $m \geq 1$ be an integer. Let $F_{W,m}$ be the product
  $\prod_{k=1}^{\mathsf{w}} F_{t_k,m}$. Write
  $F_{W,m} = \sum_{j} W_j X^j$. Then for $0 \leq j \leq m n$, the coefficient of $X^j$ in $F_{t,m}$ is
  \begin{equation}
    \sum_{\ell=\max(0, j + m(r-n))}^{j} \binom{r + \ell - 1}{\ell} W_{j - \ell}.
  \end{equation}
\end{proposition}
\begin{proof}
  Let us fix an integer $m$ and consider the dilated polytope $m Q_t$.

  For $0 \leq j \leq mn$, a point at height $j$ in $m Q_t$ can be
  described by the coordinates of its image by the projection map
  $\pi$ together with the remaining $r$ coordinates from the root vertex of
  $t$. Assume that the image by $\pi$ has height $j - \ell$, with $0
  \leq \ell \leq j$. The possible images are counted by $W_{j - \ell}$,
  which vanishes if $j - \ell > m (n - r)$ because the degree of
  $F_{w,m}$ is $m(n-r)$. Then the coordinates from the root vertex
  describe a point in $\NN^r$ with height $\ell$. This set is
  well-known to be counted by $\binom{r+\ell-1}{\ell}$.

  As this is a bijection, we get the expected sum over $\ell$.
\end{proof}

This allows to compute the usual Ehrhart polynomial by interpolation
of the values of $F_{t,m}$ at $X=1$ with respect to $m$.

For the running example, one finds that $E_t$ is
\begin{equation*}
  \frac{1}{12} \cdot (2u + 1) \cdot {(u + 1)}^3 \cdot (83 u^2 + 70 u + 12),
\end{equation*}
and so the volume is the leading coefficient $83/6$.

\begin{conjecture}
  For any arbor $t$, all roots of the Ehrhart polynomial of $Q_t$
  are negative real numbers in the interval $[-1,0]$, excluding $0$.
\end{conjecture}

Indeed $0$ is not a root because the constant term is always $1$. This
conjecture has been checked for all arbors of size at most $12$.

Note that this conjecture would imply that all coefficients of the
Ehrhart polynomial are positive. A lattice polytope with this property
is said to be \textit{Ehrhart positive}. Postnikov proved
in~\cite{post09} that all $Y$-generalized permutohedra are Ehrhart
positive, see also~\cite[\S 3.1]{liu_positivity} and~\cite[\S
6.1]{ferroni_higashitani}. By \cref{y_gen}, one can conclude that the
polytope $Q_t$ is Ehrhart positive.


\subsection{Volume}

Let us now consider the volume of the polytope $Q_t$. This could be
computed using the leading coefficient of the Ehrhart polynomial.
Instead, we refine here the volume into a volume function $V_t(h)$,
that records the volume of every slice of $Q_t$ according to the
height function $\haut$. This function vanishes outside the segment
$[0, n]$ where $n$ is the size of $t$.

This volume-function could be defined \textit{verbatim} for any polytope in
$\RR_{\geq 0}^n$. In this generality, the volume-function of a direct
product of polytopes is the convolution of the individual volume-functions.

For the arbor $t$ on one element,
\begin{equation}
  V_{t}(h) = 1,
\end{equation}
for $0 \leq h \leq 1$ and $0$ otherwise.

Let $t$ be an arbor on a set of cardinality $n \geq 2$. Let
$(t_1,\ldots,\tw)$ be the sub-trees of the root vertex of $t$. Let
$r$ be the size of the root vertex of $t$.

Let $\star$ denote the convolution product of functions on
$\RR_{\geq 0}$. Let $\trunc_n$ denote the operator that truncates
a function to the segment $[0,n]$, extended by zero outside of this segment.

\begin{proposition}
  The volume function $V_t$ is given by
  \begin{equation}
    \trunc_n (V_{t_1} \star V_{t_2} \star \cdots \star V_{\tw} \star W_{n,r}),
  \end{equation}
  where $W_{n,r}$ is the truncated function $\trunc_n\left(\frac{h^{r-1}}{(r-1)!}\right)$.
\end{proposition}
\begin{proof}
  Recall that a point in $Q_t$ is determined by a sequence of points
  $e_k$ in $Q_{t_k}$ by restriction and by a point in $\RR^r$ for the
  coordinates with indices in the root vertex of $t$. The additional
  constraint is the value of the height. Using multiplicativity, one
  checks that the volume of a slice in $Q_t$ is the expected convolution,
  except maybe for heights above $n$. This is fixed by the truncation
  operator wrapping the convolution products.
\end{proof}

The volume function allows to compute the volume by
integration from $0$ to $n$. This matches the leading coefficient of
the Ehrhart polynomial.

In general, the function $V_t$ is continuous and piecewise-polynomial with possible
change of polynomial at integers. In the running example for instance,
it is given by $h^5/120$ between $0$ and $1$, by three complicated
polynomials of degree $5$ between $1$ and $4$ and by $2 h - 14/3$ between $4$ and
$6$. Integrating from $0$ to $6$, one computes the volume to be $83/6$
as expected.

\subsection{Laplace transform of volume}

The volume function $V_t$ is in general a rather complicated
piecewise-polynomial function of the variable $h$. It is more
compactly encoded by its Laplace transform, defined by
\begin{equation}
  L_t(v) = \int_{0}^{\infty} e^{-v h} V_t(h) dh.
\end{equation}
Note that the value at $v=0$ recovers the volume of $Q_t$.

The Laplace transform sends the convolution product to
the product. It therefore remains to understand the truncation operator
$\trunc_n$ after the transform. We will only need the action of
$\trunc_n$ on the functions $e_{k, \ell}(v) = v^{-k} e^{-\ell v}$ with
$k > 0$ and $\ell \geq 0$.

\begin{lemma}\label{laplace_connu}
  For integers $k\geq 0$ and $\ell \geq 0$, let $f_{k,\ell}$ be the function on
  $\RR_{\geq 0}$ defined by $\frac{{(h-\ell)}^k}{k!}$ if $h > \ell$ and $0$
  otherwise. The Laplace transform of $f_{k,\ell}$ is $v^{-k-1} e^{-\ell v}$.
\end{lemma}
This statement is easy to prove by standard properties of the Laplace
transform, see \cite[Table 1.14.4]{DLMF}.

\begin{lemma}
  If $\ell \geq n$, then $\trunc_n(e_{k, \ell}) = 0$. Otherwise
  \begin{equation}
    \label{truncator}
    \trunc_n(e_{k + 1, \ell}) = e_{k + 1, \ell} - \sum_{j=0}^{k} \frac{{(n-\ell)}^{k-j}}{(k-j)!}  e_{j+1, n}.
  \end{equation}  
\end{lemma}
\begin{proof}
  Let $g_n(h)$ be the indicatrix function of the interval $[0,n]$. As
  $g_n = f_{0,0} - f_{0,n}$, its Laplace transform is
  $\lap g_n = \frac{1}{v}(1-e^{-nv})$. The truncation operator acts on
  functions of $h$ as the pointwise product with $g_n$. Because
  $f_{k,\ell}$ has support on $[\ell, \infty]$, this pointwise product
  $f_{k, \ell} g_n$ is the zero function if $\ell \geq n$. So
  $\trunc_n(e_{k, \ell})$ is zero in this case.

  Otherwise, assume $\ell < n$. Let us start from
  \begin{equation*}
    f_{k,\ell} g_n = f_{k,\ell}  - f_{0,n} f_{k,\ell}
  \end{equation*}
  and compute the Laplace transform:
  \begin{equation*}
    \frac{e^{-\ell v}}{v^{k+1}} - \int_{n}^{\infty} \frac{{(h-\ell)}^k}{k!} e^{-vh} dh.
  \end{equation*}
  The second term becomes, first replacing $h-\ell$ by $(h-n)+(n-\ell)$ and expanding,
  \begin{equation*}
    \sum_{j=0}^{k} \frac{{(n-\ell)}^{k-j}}{(k-j)!} \frac{e^{-nv}}{v^{j+1}}.
  \end{equation*}
  This proves the expected formula.
\end{proof}

One can therefore compute the Laplace transform $L_t(v)$ of the volume
function $V_t$ by first computing by induction a polynomial $L_t(E,V)$
in variables $E$ and $V$ using the following rules, and then replace
$E$ by $e^{-v}$ and $V$ by $1/v$. 

The initial value $L_t(E,V)$ for the arbor of size $1$ is $-V E + V$.

The truncation formula \eqref{truncator} translates into a similar
formula for the action of $\trunc_n$ on polynomials in $E$ and $V$.

Let $t$ be an arbor on a set of cardinality $n \geq 2$. Let
$(t_1,\ldots,\tw)$ be the sub-trees of the root vertex of $t$. Let $r$
be the size of the root vertex of $t$.

\begin{proposition}
  The polynomial $L_t(E,V)$ is $\trunc_n (\prod_{k=1}^{\mathsf{w}} L_{t_k}(E,V) \cdot \trunc_n(V^r))$.
\end{proposition}
\begin{proof}
  This follows from the preceding statements. The function $W_{n,r}$ has Laplace transform equal to $\trunc_n\left(v^{-r}\right)$ by \Cref{laplace_connu}.
\end{proof}

In the running example, this gives
\begin{equation*}
  -V^{6} E^{4} + 2 V^{6} E^{3} - 2 V^{5} E^{4} + 4 V^{5} E^{3} - 2 V^{2} E^{6} - 2 V^{6} E - 2 V^{5} E^{2} - \frac{22}{3} V E^{6} + V^{6}.
\end{equation*}

To recover the volume, one needs to take the limit at $v=0$ of $L_t(e^{-v},1/v)$.

\subsection{Zeta polynomial}

Let us now consider the Zeta polynomial of the poset $P_t$. This is a
polynomial $Z_t$ in the variable $u$ whose value at every integer
$m \geq 2$ is the number of chains
$e_1 \leq e_2 \leq \cdots \leq e_{m-1}$ in $P_t$.

We will need a mild refinement of this usual Zeta polynomial. Let
$\mathsf{Z}_t(u,X)$ be the polynomial in $u$ and $X$ such that for
every integer $m \geq 2$,
\begin{equation}
  \mathsf{Z}_t(m, X) = \sum_{e_1\leq\cdots\leq e_{m-1}} X^{\haut(e_{m-1})}
\end{equation}
where $\haut$ is the height of an element. The existence of this
polynomial will follow from the induction below. Setting $X=1$ forgets
about the height and gives the usual Zeta polynomial.

Note that if we consider a product of two graded posets endowed with the
sum of height functions, then the definition of $\mathsf{Z}$ is
naturally multiplicative, as chains are then just pairs of chains of the same length.

For the arbor $t$ on one element, one finds that
\begin{equation}
  \mathsf{Z}_{t}(u, X) = 1 + (u - 1) X.
\end{equation}

Let now $t$ be an arbor on a set of cardinality $n \geq 2$. Let
$(t_1,\ldots,\tw)$ be the sub-trees of the root vertex of $t$. Let
$r$ be the size of the root vertex of $t$.

\begin{proposition}
  Let $\mathsf{Z}_W$ be the product of $\prod_{k=1}^{\mathsf{w}} \mathsf{Z_{t_k}}$. Write
  $\mathsf{Z}_W = \sum_{j} W_j X^j$. Then for $0 \leq j \leq n$, the coefficient
  of $X^j$ in $\mathsf{Z}_t$ is
  \begin{equation}
    \sum_{\ell=\max(0, j-n+r)}^{j} \binom{r (u - 1) + \ell - 1}{\ell} W_{j - \ell}.
  \end{equation}
\end{proposition}
\begin{proof}
  Let us fix an integer $m \geq 2$ and consider a chain
  $e_1 \leq \cdots \leq e_{m-1}$ in $P_t$ with $\haut(e_{m-1}) =
  j$. Because the interval $[0,e_{m-1}]$ is a product of total orders, one
  can decompose the chain as a product of chains of the same length by
  separating the coordinates in two disjoint subsets: $A$ for those in
  the root vertex of $t$ and $B$ for the others. Let $\ell$ be the sum
  of coordinates of $e_{m-1}$ in $A$ and $j - \ell$ be the sum of
  coordinates of $e_{m-1}$ in $B$. Note that $j - \ell \leq n - r$ by
  the defining inequalities for the sub-trees $t_1,\ldots,\tw$.

  On the set $B$, the restriction of the chain is a product of chains
  in the sub-trees $t_1,\ldots,\tw$ with maximal height $j -
  \ell$. Such chains are counted by $W_{j-\ell}$ evaluated at
  $u=m$. On the set $A$, the restriction of the chain is a chain in
  the poset $\NN^r$ ordered coordinate-wise, with maximal height equal
  to $\ell$. By \Cref{zeta_B} in \Cref{sec:B}, chains of length $m-1$
  in the poset $\NN^r$ with maximal height at most $\ell$ are counted
  by $\binom{r (m-1) + \ell}{\ell}$. Therefore, by taking the
  difference of values for $\ell$ and $\ell-1$, the possible
  restricted chains on the set $A$ are counted by
  $\binom{r (m-1) + \ell - 1}{\ell}$.

  This decomposition is a bijection, hence the expected formula holds
  for every $m \geq 2$. Hence it holds as an equality of polynomials in $u$.
\end{proof}

In the running example, one gets
\begin{multline*}
  \mathsf{Z}_t = \frac{1}{90} \cdot \Bigl(
{\left(802 \, u^{4} - 1839 \, u^{3} + 1598 \, u^{2} - 621 \, u + 90\right)} {\left(2 \, u - 1\right)} {\left(u - 1\right)} X^{6} \\
  + 6 \, {\left(53 \, u^{2} - 77 \, u + 30\right)} {\left(4 \, u - 3\right)} {\left(2 \, u - 1\right)} {\left(u - 1\right)} X^{5} \\
+ 45 \, {\left(68 \, u^{3} - 156 \, u^{2} + 119 \, u - 30\right)} {\left(u - 1\right)} X^{4}\\
+ 60 \, {\left(11 \, u - 10\right)} {\left(4 \, u - 3\right)} {\left(u - 1\right)} X^{3} \\
 + 90 \, {\left(17 \, u - 15\right)} {\left(u - 1\right)} X^{2} + 540 \, {\left(u - 1\right)} X + 90 \Bigr)
\end{multline*}
and so the usual Zeta polynomial is
\begin{equation*}
  \frac{1}{90} \cdot u \cdot (2 u - 1) \cdot (802 u^4 - 1369 u^3 + 893 u^2 - 266u + 30).
\end{equation*}

It seems that the set of arbors $t$ having the property that the
Zeta polynomial of $P_t$ factors completely over $\QQ$ consists exactly of
arbors of type $\TA$ considered in \Cref{sec:A}, arbors of type $\TB$
considered in section \Cref{sec:B} and one of the two families
considered in \Cref{sec:halohedra}.

\subsection{$M$-triangle}\label{subsec:M}

Let us now consider the $M$-triangle of the poset $P_t$. This is the
polynomial $M_t(X,Y)$ defined by \eqref{defiM} in
\Cref{app:triangles}, using values of the Möbius function.

In order to give a recursive formula, we need to compute instead the following polynomial
\begin{equation}
  \label{def_Kt}
  K_t(X,Y) = \sum_{a \in P_t}  X^{\nz(a)} Y^{\haut(a)},
\end{equation}
where $\nz(a)$ is the number of non-zero coordinates in $a$. Note that
the exponent of $X$ is smaller than or equal to the exponent of $Y$ in
any monomial in this formula.

Note that the polynomial $K_t$ is clearly multiplicative under product
for posets of the type that we are considering.

For the arbor $t$ on one element,
\begin{equation}
  K_{t} = 1 + XY.
\end{equation}

Let $t$ be an arbor on a set of cardinality $n \geq 2$. Let
$(t_1,\ldots,\tw)$ be the sub-trees of the root vertex of $t$. Let
$r$ be the size of the root vertex of $t$.

\begin{proposition}\label{recurrence_K}
  Let $W = \prod_k K_{t_k}$ and write
  $W = \sum_{i,j} W_{i,j} X^i Y^j$. Then for
  $0 \leq j \leq k \leq n$, the coefficient of $X^j Y^k$ in $K_t$ is
  \begin{equation}
    \sum_{\ell=\max(0,j-n+r)}^{\min(j, r)} \sum_{m=\max(\ell, k -n + r)}^{k+\ell-j} 
    \binom{r}{\ell} \binom{m-1}{m-\ell} W_{j-\ell,k-m}.
  \end{equation}
\end{proposition}
\begin{proof}
  Any element $a$ of $P_t$ can be described as an element $\pi(a)$ of
  $\prod_k P_{t_k}$ and $r$ additional coordinates from the root of
  $t$ for the fiber of $\pi$. Let us assume that $a$ has $j$ non-zero
  coordinates and height $k$, with $0 \leq j \leq k \leq n$. Assume
  now that $\pi(a)$ has $j-\ell$ non-zero coordinates and height $k-m$
  with $0 \leq \ell \leq j$ and $0 \leq m \leq k$. This also requires
  that $0 \leq j - \ell \leq k - m \leq n - r$.
  
  Then the additional $r$ coordinates must sum to $m$ and among them
  exactly $\ell$ must be non-zero. This implies that $\ell \leq r$ and
  $\ell \leq m$. The choice of these non-zero coordinates is counted
  by a binomial $\binom{r}{\ell}$. Once this choice is fixed, there
  remains to choose $\ell$ strictly positive numbers summing to $m$,
  which amounts to choose $\ell-1$ cuts among $m-1$ possible inner cuts
  in a sequence of length $m$.

  As this defines a bijection, one gets the expected double sum, where
  the projection $\pi(a)$ contributes by the term $W_{j-\ell,k-m}$ and
  the coordinates from the root contribute the product
  $\binom{r}{\ell} \binom{m-1}{m-\ell}$. The precise bounds in the
  double sum follows from the necessary inequalities as previously
  stated.
\end{proof}

Here is a slightly simplified version of \Cref{recurrence_K} that will
be useful.
\begin{proposition}\label{recurrence_K1}
  Let $W = \prod_k K_{t_k}$ and write
  $W = \sum_{i,j} W_{i,j} X^i Y^j$. Then for
  $0 \leq j \leq n$, the coefficient of $X^j$ in $K_t(X,1)$ is
  \begin{equation}
    \sum_{\ell=\max(0,j-n+r)}^{\min(j, r)} \sum_{k'=j-\ell}^{n-r} 
    \binom{r}{\ell} \binom{n-k'}{\ell} W_{j-\ell,k'}.
  \end{equation}
\end{proposition}
\begin{proof}
  Starting from \Cref{recurrence_K} and summing over $k$, one gets
  \begin{equation}
    \sum_{k,\ell,m} \binom{r}{\ell}\binom{m-1}{m-\ell} W_{j-\ell,k-m}
  \end{equation}
  over the appropriate range. By first changing variable to
  $k'=k-m$, one can perform the inner summation over $m$ from $\ell$ to $n-k'$.
\end{proof}

The polynomial $K_t$ determines the $M$-triangle $M_t$ as follows.

\begin{proposition}
  For every arbor $t$, 
  \begin{equation}
    \label{M_from_K}
    M_t(X, Y) = K_t(1 - 1 / X, XY).
  \end{equation}
\end{proposition}
\begin{proof}
  First, write the $M$-triangle as a double sum
  \begin{equation*}
    \sum_{b \in P_t} Y^{\haut(b)} \sum_{a \leq b} \mu(a,b) X^{\haut(a)}.
  \end{equation*}
  Recall that the interval $[0,b]$ is a product of total orders whose
  lengths are the non-zero coordinates of $b$ plus $1$. It follows
  that the inner sum is
  \begin{equation*}
    \sum_{a \leq b} \mu(a,b) X^{\haut(a)} = X^{\haut(b)} {(1-1/X)}^{\nz(b)}.
  \end{equation*}
  The result then follows from the definition of $K_t$.
\end{proof}

In our running example, the $M$-triangle is
\begin{equation*}
\left(\begin{array}{rrrrrrr}
1 & -18 & 113 & -334 & 506 & -380 & 112 \\
-6 & 58 & -208 & 352 & -284 & 88 &  \\
15 & -92 & 201 & -188 & 64 &  &  \\
-20 & 78 & -98 & 40 &  &  &  \\
15 & -34 & 19 &  & &  &  \\
-6 & 6 &  &  &  &  &  \\
1 &  &  &  &  &  & 
\end{array}\right)
\end{equation*}
where the south-west to north-east diagonal records the number of elements in $P_t$ according to height.

\subsection{$f$-vector and $h$-vector}

As all intervals in the poset $P_t$ are products of total orders, the Hasse
diagram of $P_t$ can be viewed as the skeleton of a cubical complex. Let
us compute its $f$-vector, the polynomial in one variable $X$ where
the coefficient of $X^k$ is the number of $k$-dimensional cubical faces.

Every cubical face is determined by an interval $[a,b]$ in $P_t$. We
will gather them according to their top element $b$. When $b$ has
$\nz(b)$ non-zero coordinates, it contributes by the faces of a
$\nz(b)$-dimensional hypercube that are incident to a fixed vertex of this hypercube. The $f$-vector is therefore
\begin{equation}
  \label{def_ft}
  f_t(X) = \sum_{b \in P_t} {(1+X)}^{\nz(b)}.
\end{equation}
By~\eqref{def_Kt}, this can be computed from the polynomial $K_t$ as
\begin{equation}
  f_t(X) = K_t(1+X, 1).
\end{equation}

The $h$-vector of the arbor $t$ is defined as the generating polynomial
\begin{equation*}
    \label{def_ht}
    h_t(X) = \sum_{b \in P_t} {X}^{\nz(b)}.
\end{equation*}
It satisfies $h_t(X) = f_t(X-1) = K_t(X, 1)$. By~\eqref{M_from_K},
this implies that
\begin{equation}
  \label{h_from_M}
  h_t(X) = M_t\left(\frac{1}{1-X},1-X\right).
\end{equation}

\smallskip

In our running example, one finds that the $f$-vector of $P_t$ is
\begin{equation*}
  f_t(X) = X^{6} + 24X^{5} + 186X^{4} + 654X^{3} + 1152X^{2} + 990X + 330
\end{equation*}
and its $h$-vector is
\begin{equation*}
  h_t(X) = X^{6} + 18X^{5} + 81X^{4} + 130X^{3} + 81X^{2} + 18X + 1,
\end{equation*}
which is nicely palindromic.


\section{Ehrhart-Zeta duality for linear arbors}\label{sec:EhrhartZeta}

This section proposes a conjectural duality exchanging Ehrhart
polynomials and Zeta polynomials.

Let $t$ be an arbor such that each vertex of $t$ has at most one
sub-tree. This is called a \textit{linear arbor}. Still the vertices can
contain several elements.

For a linear arbor $t$, let $\rev(t)$ be the reverse of $t$, defined as
the linear arbor in which the vertices are in the opposite order. This
amounts to choose the root vertex of $\rev(t)$ to be the unique leaf of
$t$. Reversal is an involution on linear arbors.

For example, here are two arbors $t$ and $\rev(t)$, mutually reverse,
with their roots on the left.
\begin{center}
\begin{tikzpicture}
[grow=right]
\tikzstyle{every node}=[draw,shape=circle]

\node [fill=racines,thick] {$a,f$}
  child {node {$b,e$}
  child {node {$d$}
  child {node {$c$}}}}
;
\end{tikzpicture}
\quad\quad
\begin{tikzpicture}
[grow=right]
\tikzstyle{every node}=[draw,shape=circle]

\node [fill=racines,thick] {$c$}
  child {node {$d$}
  child {node {$b,e$}
  child {node {$a,f$}}}}
;
\end{tikzpicture}
\end{center}

\begin{conjecture}\label{EZ_conjecture1}
  For every linear arbor $t$, the Ehrhart polynomial $E_t(u)$ of the polytope $Q_t$ is equal
  to the shifted Zeta polynomial $Z_{\rev(t)}(u+1)$ of the poset $P_{\rev(t)}$.
\end{conjecture}

By looking at the leading coefficients, \Cref{EZ_conjecture1} would identify the normalized volume of $Q_t$ to
the number of maximal chains of $P_{\rev(t)}$. Recall that in
dimension $n$ the normalized volume is $n!$ times the volume.

By the special case of the value at $u=1$, the posets of the arbors
$t$ and $\rev(t)$ should share the same number of elements. Something
finer seems to hold.

\begin{conjecture}\label{EZ_conjecture2}
  For every linear arbor $t$, the $h$-vectors of $t$ and $\rev(t)$ are equal.
\end{conjecture}

\smallskip

For example, for the two arbors above, both with $501$ elements, one gets
\begin{alignat*}{2}
  E_t &= \frac{u+1}{720} \cdot (20167 u^5 + 57230 u^4 + 62225 u^3 + 32170 u^2 + 7848 u + 720),\\
  Z_t &= \frac{u}{60} \cdot (2 u - 1) \cdot (898 u^4 - 1611 u^3 + 1002 u^2 - 249 u + 20),\\
  E_{\rev(t)} &= \frac{u+1}{60} \cdot (2 u + 1) \cdot (898 u^4 + 1981 u^3 + 1557 u^2 + 514 u + 60),\\
  Z_{\rev(t)} &= \frac{u}{720} \cdot (20167 u^5 - 43605 u^4 + 34975 u^3 - 12795 u^2 + 2098 u - 120).\\
\end{alignat*}

The volume of $Q_t$ is $20167/720$ and the number of maximal chains of
$P_t$ is $21552$. Conversely, the volume of $Q_{\rev(t)}$ is $21552/720$ and
the number of maximal chains of $P_{\rev(t)}$ is $20167$. Both $t$ and $\rev(t)$ have $h$-vector
\begin{equation*}
  x^6 + 23 x^5 + 122 x^4 + 209 x^3 + 122 x^2 + 23 x + 1.
\end{equation*}

Both conjectures have been checked by computer for all linear arbors
of size at most $11$. We will not purse these conjectures further
here. One may wonder if the first one can be explained by the
existence of unimodular triangulations.

\section{Transmutation of $M$-triangles}\label{sec:transmutation}

Recall from \Cref{app:triangles} that the $M$-triangle of a graded
poset is the following polynomial:
\begin{equation*}
  m = \sum_{a \leq b} \mu(a, b) X^{\haut(a)} Y^{\haut(b)},
\end{equation*}
where $\mu$ is the Möbius function of the poset.

Let us introduce the following operation, called transmutation:
\begin{equation}
  \label{transmutation}
  \overline{m}(X,Y) = m\left(\frac{1-Y}{1-XY},1-XY\right).
\end{equation}
Transmutation is an involution, as can be checked directly. We will
explain later where it comes from.

\smallskip

Note that for the posets $P_t$, the $M$-triangle $M_t$ can be computed from
the simpler polynomial $K_t$ as defined in \eqref{def_Kt}. It follows
that the transmuted $M$-triangle is
\begin{equation}
  \label{transmuted_M_from_K}
  \overline{M}_t(X,Y) = K_t\left(\frac{Y(X-1)}{1-Y}, 1-Y\right).
\end{equation}

\smallskip

As explained at length in \Cref{subsec:motivation}, one of the main
motivations for the article is the following hope: there
should exist another graded poset $\NC_t$ whose $M$-triangle is
$\overline{M}_t$.

\smallskip

For example, consider the arbor $t$ equal to
\begin{center}
\begin{tikzpicture}
[grow=right]
\tikzstyle{every node}=[draw,shape=circle]

\node [fill=racines,thick] {$a$}
  child {node {$b$}
  child {node {$c$}}}
;
\end{tikzpicture}.
\end{center}
Then the poset $P_t$ has $14$ elements. Its $M$-triangle and transmuted $M$-triangle are
\begin{equation*}
  \left(\begin{array}{rrrr}
          -1 & 6 & -10 & 5 \\
          3 & -8 & 5 &  \\
          -3 & 3 &  &  \\
          1 &  & & 
        \end{array}\right)\quad \text{and}\quad\left(\begin{array}{rrrr}
-5 & 10 & -6 & 1 \\
10 & -16 & 6 &  \\
-6 & 6 &  &  \\
1 &  &  & 
\end{array}\right),
\end{equation*}
displayed as matrices of coefficients, with the constant term at
the bottom left.

One can check that the transmuted $M$-triangle is the $M$-triangle of
the classical lattice of noncrossing-partitions of a set of size $4$,
as defined in \cite{kreweras}.

\smallskip

\begin{proposition}\label{transmuted_h}
  The diagonal of the transmuted $M$-triangle is the $h$-vector of $t$.
\end{proposition}
\begin{proof}
  Formula \eqref{transmuted_M_from_K} expresses the transmuted
  $M$-triangle as
  \begin{equation*}
    \sum_{a \in P_t} {(1-Y)}^{\haut(a)- \nz(a)} {(Y(X-1))}^{\nz(a)}.
  \end{equation*}
  In every monomial in this expression, the power of $Y$ is larger than
  or equal to the power of $X$. The restriction to monomials with equal powers is
  \begin{equation*}
    \sum_{a \in P}{(YX)}^{\nz(a)},
  \end{equation*}
  which is equal to $K_t(XY,1)$ encoding the $h$-vector of $t$.
\end{proof}

\smallskip

Let us now say a few words about the origin of the transmutation operation.

Closely linked to the $M$-triangle, there is another two-variable
polynomial called the $H$-triangle. They are related by the
transformations \eqref{M_from_H} and \eqref{H_from_M} recalled in
\Cref{app:triangles}.

These triangles and conversion rules between them first appeared in
the study of the combinatorics of finite type cluster algebras
\cite{chap_2004, chap_full}. In this context, the $M$-triangle was
considered for the noncrossing-partitions lattice attached to an
$\TA\TD\TE$-quiver. The $H$-triangle was a natural refinement of
the $h$-vector of the cluster fan for the same quiver.

One can define transmutation of $M$-triangles as the conjugation of
\textit{transposition} on $H$-triangles by the conversion maps between
$H$-triangles and $M$-triangles. Here transposition just means, viewing
$H$-triangles as triangular matrices in the appropriate way, the
transposition of matrices.

Note that the $H$-triangle can be interpreted as a refined
decomposition of the homology groups of a smooth toric variety. In
this context, the transposition has no clear meaning, as it exchanges
the homological grading with the other additional grading.

\section{Unary-binary arbors and quadrangulations}\label{sec:unarybinary}


Let us say that an arbor is unary-binary if 
\begin{itemize}
\item every vertex has one element and
\item every vertex has at most 2 sub-trees.
\end{itemize}
We will talk about vertices of valency $0$, $1$ or $2$.

For example, here is a unary-binary arbor:
\begin{center}
\begin{tikzpicture}
[grow=right]
\tikzstyle{every node}=[draw,shape=circle]

\node [fill=racines,thick] {$a$}
  child {node {$b$}}
  child {node {$c$}
  child {node {$d$}}}
;
\end{tikzpicture},
\end{center}
where the valencies of $a,b,c,d$ are $2,0,1,0$.

Given such an arbor $t$, let us choose a planar embedding and assume
that the arbor is drawn with the root on the left and growing to the
right as above.

Using this picture, one can define a quiver $G_t$ as follows. The
vertices of $G_t$ are the vertices of $t$. The edges of $G_t$ are the
edges of $t$ suitably oriented. Let $v$ be a vertex $v$ of valency
$2$. Let $\alpha_v$ and $\beta_v$ the two edges between $v$ and its
$2$ sub-trees, from top to bottom. We orient $\alpha_v$ towards $v$ and
$\beta_v$ from $v$ outwards. Other edges are oriented arbitrarily.

We now consider $G_t$ as a quiver-with-relations for the relations
$\alpha_v \beta_v = 0$ for all vertices $v$ of valency $2$.

For instance, by choosing to orient the remaining arrow from $d$ to $c$, the
previous example gives the quiver-with-relations
\begin{center}
\begin{tikzpicture}[->,>=stealth,auto,node distance=2cm,
  thick,main node/.style={circle,draw=black}]
  \node[main node] (1) {$d$};
  \node[main node] (2) [right of=1] {$c$};
  \node[main node] (3) [right of=2] {$a$};
  \node[main node] (4) [right of=3] {$b$};

  \draw [->] (1) -- (2);
  \draw [->] (2) -- node {$\alpha$}(3);
  \draw [->] (3) -- node {$\beta$} (4);
  \draw [dashed,-] (2) to [out=45,in=135] (4);
\end{tikzpicture}.
\end{center}

We refer the reader to \cite{gentle2010, FDatlas} for the definition and properties of
gentle quivers and algebras. Essentially this means that at every
vertex, there are at most two incoming arrows $\alpha, \alpha'$ and at
most two outgoing arrows $\beta, \beta'$ and there are relations
$\alpha \beta$ and $\alpha' \beta'$ when they make sense.

\begin{lemma}
  The quiver-with-relations $G_t$ is a gentle quiver.
\end{lemma}
\begin{proof}
  Every vertex is incident to at most $3$ edges and there can be no
  sink and no source vertex. If there are three incident edges, two of
  them are linked by a relation.
\end{proof}

Let us recall briefly now an important construction starting from the
gentle quiver $G_t$. This comes from $\tau$-tilting theory as
introduced in \cite{tau_tilting}. To every finite-dimensional algebra
over a field, one can associate a set of support-$\tau$-tilting
objects. These objects then naturally appear as vertices in a regular
graph called the mutation graph. It happens, for specific classes of
finite-dimensional algebras, that the mutation graph of
support-$\tau$-tilting objects is finite. That finiteness and its
consequences hold notably for the gentle algebras $G_t$ that we
consider here.

In this case, the mutation graph is in fact oriented and the Hasse
diagram of a partial order which is congruence-uniform. There is also
a simple polytope whose skeleton is the mutation graph. We refer to
\cite{ppp_tau} for these properties in the case of gentle algebras.

\smallskip

This finite mutation graph on support-$\tau$-tilting objects for the quiver-with-relations $G_t$ seems to be
related to the poset $P_t$ itself by the following conjecture.
\begin{conjecture}\label{conjecture-quad}
  The number of support $\tau$-tilting objects for the gentle
  quiver-with-relations $G_t$ is equal to the number of elements of
  the poset $P_t$.
\end{conjecture}

In fact, there should be a more precise relationship. Let us consider
the spherical simplicial complex $SC_t$ which is the polar dual of the
simple polytope attached to $G_t$ by $\tau$-tilting theory.
\begin{conjecture}
  In the simplicial complex $SC_t$, there exists a facet $C$
  such that the $F$-triangle for the pair $(SC_t, C)$ coincides with
  the $F$-triangle associated with the transmuted $M$-triangle of
  $P_t$ by the formulas of \Cref{app:triangles}.
\end{conjecture}
This can been checked concretely in small cases. It has been checked
by computer for $n \leq 8$.  

In our example $t$ as above, the poset $P_t$ has $33$ elements. One
can check that the gentle quiver $G_t$ also has $33$ support $\tau$-tilting
objects. The $F$-triangle associated with the transmuted $M$-triangle is
\begin{equation}
  \left(\begin{array}{rrrrr}
1 &  &  &  &  \\
4 & 4 &  &  &  \\
6 & 14 & 8 &  &  \\
4 & 17 & 23 & 10 &  \\
1 & 8 & 22 & 25 & 10
\end{array}\right)
\end{equation}
and it appears as expected for some facet of $SC_t$.

\subsection{Stokes polytopes and quadrangulations}

Let us present briefly an alternative description of the
same situation, replacing the general $\tau$-tilting theory by the
more specific notion of Stokes polytopes and Baryshnikov compatibility
\cite{baryshnikov, stokes}.

Let's consider a regular polygon with $2n + 4$ vertices. Fix a
quadrangulation $Q$ of this polygon, which means a set of $n$
noncrossing segments joining vertices and cutting the polygon into
quadrilateral regions. After Baryhsnikov, one can define a
compatibility relation between quadrangulations. On the set $R_Q$ of
quadrangulations compatible with $Q$, there is a mutation process,
that defines a mutation graph. Moreover, this mutation process is
oriented and defines a partial order. The mutation graph is also the
skeleton of a polytope, called the Stokes polytope for the
quadrangulation \cite{manneville_geo, manneville_serpent,
  manneville_court}.

There is a simple way to map the planar-embedded arbor $t$ into a
quadrangulation. There is also a way to define a gentle quiver from
any quadrangulation. The composite map recovers the definition of
$G_t$ from $t$. By this correspondence, the $\tau$-tilting theory
exactly matches the combinatorial description using mutations of
compatible quadrangulations.

It follows that the conjectural statements above could be rephrased entirely
using the structures (mutation graph, poset, polytope) coming from
this context instead of those coming from $\tau$-tilting theory.


\section{Arbors of type $\TA$}\label{sec:A}

In this section, we consider the case of linear arbors with no
multiple vertex, so that their shape is a simple chain. This is the
intersection of the set of unary-binary arbors considered before with
the set of linear arbors. In order to compute the Zeta polynomial and
the $M$-triangle for the poset $P_t$ in this case, it is convenient to
study a more general familly of posets.

\smallskip

In this section, we will use the following convention: for a product
$\beta_{\ell} = \prod_{r=2}^{\ell} \alpha_r$, the value at $\ell = 1$
is $1$ and the value at $\ell =0$ is $1 / \alpha_1$.

\smallskip

Let $m \geq 1$ be an integer, used as a slope parameter. Let $R_m$ be the region in the plane of coordinates
$(x,y)$ defined by the conditions
\begin{align}
  0 \leq x, \quad 0 \leq y, \quad m y < x.
\end{align}

Let now $(x,y)$ be nonnegative integers such that
\begin{align}
  \label{condition_x_y}
  m y < x.
\end{align}

We consider the set of lattice paths in $R_m$ that go from $(0,0)$ to
the fixed point $(x, y)$, by steps $(+1,0)$ (denoted by $\mX$) or
$(0,+1)$ (denoted by $\mY$). The point $(0,0)$ is the unique allowed
point on the line $my = x$ at the boundary of $R_m$.

\begin{figure}[h]
  \centering
  \includegraphics[height=3cm]{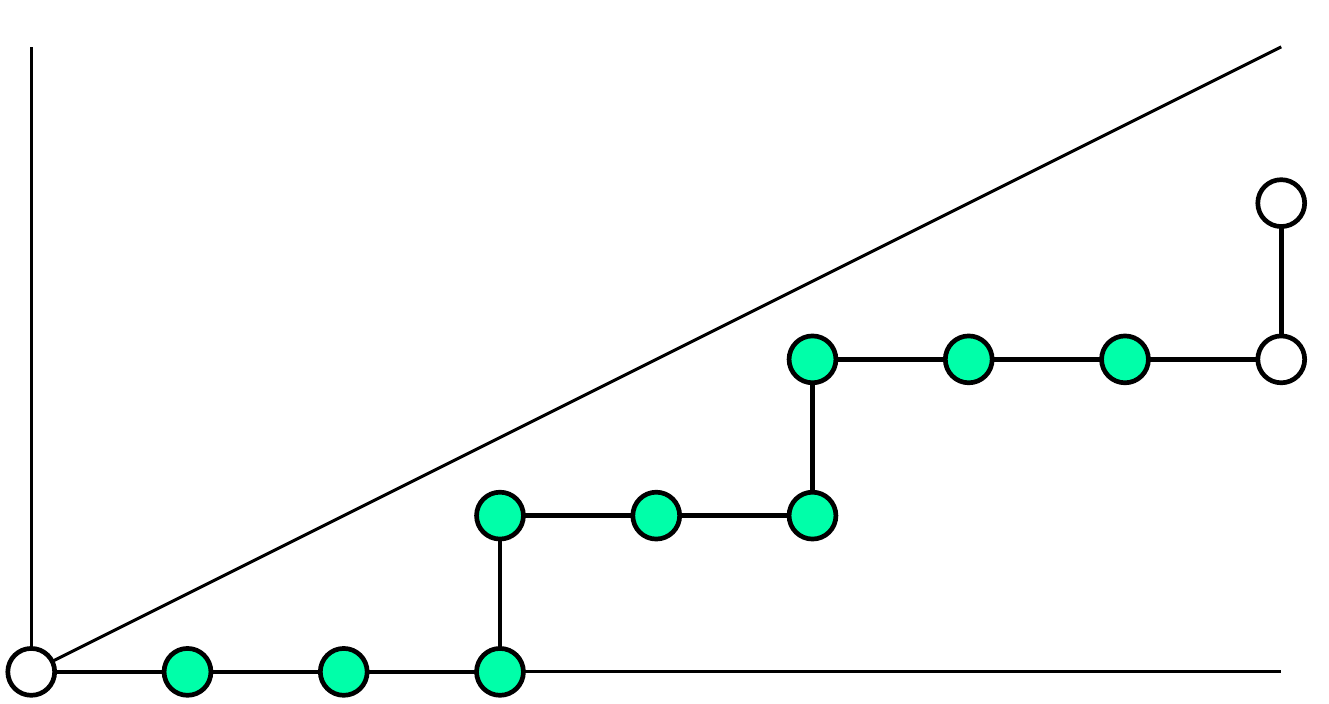}
  \caption{Example of lattice path $\mX^3\mY\mX^2\mY\mX^3\mY$ with $m=2$, $x=8$ and $y=3$.\\The word $A$
    associated with vertical steps is $(0,0,1,0,1,0,0)$.}
\end{figure}

To each lattice path $w$ in $R_m$ from $(0,0)$ to $(x,y)$, one associates a word
$A=(a_1,\ldots,a_{x-1})$, that records the lengths of sequences of consecutive $\mY$
between two letters $\mX$. The letter $a_i$ is the number of $\mY$ between the
$i$-th instance of $\mX$ and the next one. The length of $A$ is $x - 1$.

One can characterise the set of such words $A$
by the following inequalities:
\begin{equation}
  \label{condition_pente}
  m \left(\sum_{i\leq j} a_i \right) < j,
\end{equation}
for $1 \leq j \leq x - 1$ and the condition
\begin{equation}
  \label{condition_fin}
  \sum_i a_i \leq y.
\end{equation}
In particular, every $a_i$ for $i \leq m$ must vanish.

One then defines a poset $\po^F_{m,x,y}$ by term-wise comparison on
this set of words. In this notation, the exponent $F$ stands for Fuss.

The poset $\po^F_{m,x,y}$ is graded by the sum of the $a_i$, that we
will denote by $\haut$.  There is a unique minimum $\bot$, with all
$a_i$ being zero. Every interval $[\bot,{(a_i)}_i]$ is isomorphic to the
product of total orders of size $a_i + 1$.

The special case where $m=1$, $x=n+2$ and $y=n$ recovers the poset
$P_t$ attached to the linear arbor of size $n$ with no multiple
vertex. Indeed, in terms of words $(a_1,\ldots,a_{n+1})$, being
strictly below the diagonal $x=y$ amounts to the conditions
$a_1 \leq 0$, $a_1+a_2 \leq 1$, $\ldots$,
$a_1 + \cdots + a_{n+1} \leq n$. Necessarily $a_1 = 0$, so that one
finds the defining equations of the poset $P_t$ and the same term-wise
comparison.

\begin{lemma}\label{iso_poF}
  Let $y \geq 1$. The posets $\po^F_{m,my+1,y}$ and
  $\po^F_{m,my+1,y-1}$ are isomorphic.
\end{lemma}
\begin{proof}
  For all $m\geq 1$ and for all $y \geq 1$ and $x = m y + 1$, lattice
  paths in the region $R_m$ that ends at $(x,y)$ are in bijection, by
  removal of the last step $\mY$, with lattice paths in the region
  $R_m$ that ends at $(x,y-1)$.  This bijection is the identity on the
  coordinates used to define the partial order.
\end{proof}

\subsection{Zeta polynomial}

\begin{theorem}\label{zeta_A}
  The Zeta polynomial of $\po^F_{m,x,y}$ is given by
  \begin{equation}
    \label{formule_ZF}
    Z^F_{m,x,y}(u) = \frac{\bigl((x-m y -1)(u-1) + 1 \bigr)}{y!} \prod_{j=2}^{y} \bigl((x-1)(u-1)+j \bigr).
  \end{equation}
\end{theorem}
\begin{proof}
  One first notes that if $x = 1$, then $y = 0$ and the right-hand side
  is the constant polynomial $1$. As the poset
  $\po^F_{m,1,0}$ is just a singleton, the statement holds in that case.
  
  The rest of the proof is by induction based on this initial case and
  the next four lemmas.
\end{proof}

Let us denote $\Zeta^F_{m,x,y}(u)$ the Zeta polynomial of the poset $\po^F_{m,x,y}$.

\begin{lemma}
  For all $m\geq 1$ and for all $y \geq 1$ and $x = m y + 1$, there is a polynomial equality
  \begin{equation}
    \label{eq1_Zeta_F}
    \Zeta^F_{m,x,y}(u) = \Zeta^F_{m,x,y-1}(u).
  \end{equation}
\end{lemma}
\begin{proof}
  By \Cref{iso_poF}, the posets $\po^F_{m,x,y}$ and $\po^F_{m,x,y-1}$ are isomorphic.
\end{proof}

\begin{lemma}
  For all $m\geq 1$ and for all $x,y$ satisfying~\eqref{condition_x_y} with $x > m y + 1$, there is a polynomial equality
  \begin{equation}
    \label{eq2_Zeta_F}
    \Zeta^F_{m,x,y}(u) = \sum_{j=0}^{y} \Zeta^F_{m,x-1,y-j}(u) \binom{j-2+u}{j}.
    \end{equation}
  \end{lemma}
\begin{proof}
  By polynomiality of Zeta polynomials with respect to $u$, it
  suffices to check the formula when $u$ is replaced by any integer
  $q \geq 2$. We will use a bijection between chains.

  Let us assume $q \geq 2$ and consider a chain
  $A_1 \leq A_2 \leq \cdots \leq A_{q-1}$ in $\po^F_{m,x,y}$. We
  denote by $a_{\ell,i}$ for $1 \leq i \leq x-1$ the coordinates of the
  lattice path $A_\ell$. Note that the chain condition implies that
  $\sum_{i\leq k} a_{\ell,i}$ is growing with $\ell$ for every fixed $k$.

  Let $j = a_{q-1,x-1}$. The possible values of $j$ are all integers between $0$ and $y$.

  For $1 \leq \ell \leq q-1$, we define $\widehat{A}_\ell$ by forgetting
  the coordinate $a_{\ell, x-1}$ in $A_\ell$. Then every
  $\widehat{A}_\ell$ defines an element in $\po^F_{m,x-1,y-j}$.
  Indeed, the conditions \eqref{condition_pente} remain trivially true
  and the condition \eqref{condition_fin} is implied by
  the definition of $j$ and the growth of $\sum_{i<x-1} a_{\ell,i}$ as a function of
  $\ell$:
  \begin{equation*}
    \sum_{i < x-1} a_{\ell,i} + j \leq \sum_{i < x-1} a_{q-1,i} + a_{q-1,x-1} \leq y.
  \end{equation*}
  We therefore get a chain
  $\widehat{A}_1 \leq \widehat{A}_1 \leq \widehat{A}_2 \cdots \leq
  \widehat{A}_{q-1}$ in $\po^F_{m,x-1,y-j}$.

  To build back the original chain $A_1 \leq A_2 \leq \cdots \leq A_{q-1}$
  from this truncated chain, one has to choose appropriate
  values $a_{\ell,x-1}$ for $1 \leq \ell \leq q-1$.
  
  To recover $A_{q-1}$ in $\po^F_{m,x,y}$,
  one can only choose  $a_{q-1,x-1} = j$. For this choice, the
  conditions \eqref{condition_pente} and \eqref{condition_fin} hold
  for $A_{q-1}$.

  Now choose the other integers $a_{\ell,x-1}$ such that
  $0 \leq a_{1,x-1} \leq \cdots \leq a_{q-1,x-1} = j$.
  Then the conditions \eqref{condition_pente} and
  \eqref{condition_fin} for $A_\ell$ result from those for
  $A_{q-1}$. So we have indeed built back a chain in $\po^F_{m,x,y}$.

  For each fixed $j$, the number of possible choices for the whole
  increasing sequence of integers $a_{\ell,x-1}$
  ($1 \leq \ell < q - 1$) is therefore a binomial coefficient
  $\binom{j+q-2}{j}$.
\end{proof}

\begin{lemma}
  For all $m\geq 1$ and for all $y \geq 1$ and $x = m y + 1$,
  there is a polynomial equality
  \begin{equation}
    Z^F_{m,x,y}(u) = Z^F_{m,x,y-1}(u).
  \end{equation}
\end{lemma}
\begin{proof}
  This is an easy computation using \eqref{formule_ZF}. 
\end{proof}

The following lemma is a computation using hypergeometric functions.

\begin{lemma}
  For all $m\geq 1$ and for all $x,y$ satisfying
  \eqref{condition_x_y} with $x \geq m y + 1$,
  there is a polynomial equality
  \begin{equation}
    \label{eq_ZF}
    Z^F_{m,x,y}(u) = \sum_{j=0}^{y} Z^F_{m,x-1,y-j}(u) \binom{j+u-2}{j}.
  \end{equation}
\end{lemma}
\begin{proof}
  In the right hand side of \eqref{eq_ZF}, let us use the formula
  \eqref{formule_ZF}. One finds
  \begin{equation*}
    \sum_{j=0}^{y} \binom{j+u-2}{j} \frac{\bigl((x-m (y-j) -2)(u-1) + 1 \bigr)}{(y-j)!} \prod_{\ell=2}^{y-j} \bigl((x-2)(u-1)+\ell \bigr).
  \end{equation*}

  Using the shorthands $\bu=u-1$ and $\bx=x-2$, this becomes
  \begin{equation*}
    \sum_{j=0}^{y} \binom{j+\bu-1}{j} \frac{\bigl((\bx -m (y-j)) \bu + 1 \bigr)}{(y-j)!} \prod_{\ell=2}^{y-j} \bigl( \bx \bu+\ell \bigr).
  \end{equation*}

  One can separate into two sums:
  \begin{equation*}
    \sum_{j=0}^{y} \binom{j+\bu-1}{j} \frac{1}{(y-j)!} \prod_{\ell=1}^{y-j} \bigl( \bx \bu+\ell \bigr) -m \bu   \sum_{j=0}^{y} \binom{j+\bu-1}{j} \frac{1}{(y-j-1)!} \prod_{\ell=2}^{y-j} \bigl( \bx \bu+\ell \bigr).
  \end{equation*}

  In each of these sums, one can replace the summation upper bound $y$
  by $+\infty$ and get an hypergeometric expression:
  \begin{equation*}
    \frac{1}{y!}\prod_{\ell=1}^{y} \bigl( \bx \bu+\ell \bigr) \cdot \pFq{2}{1}{\bu,-y}{-y-\bx\bu} - \frac{m \bu}{(y-1)!} \prod_{\ell=2}^{y} \bigl( \bx \bu+\ell \bigr) \cdot \pFq{2}{1}{\bu,-y+1}{-y-\bx\bu}.
  \end{equation*}

  Then using the Chu-Vandermonde formula for each term, one gets
  \begin{equation*}
    \frac{1}{y!}\prod_{\ell=1}^{y} \bigl( \bx \bu+\ell \bigr) \frac{{(-y-\bx\bu-\bu)}_y}{{(-y-\bx\bu)}_y} - \frac{m \bu}{(y-1)!} \prod_{\ell=2}^{y} \bigl( \bx \bu+\ell \bigr) \frac{{(-y-\bx\bu-\bu)}_{y-1}}{{(-y-\bx\bu)}_{y-1}},
  \end{equation*}
  where one has used some Pochhammer symbols. This can be simplified as
  \begin{equation*}
    \frac{1}{y!}\prod_{\ell=1}^{y} \bigl( \bx \bu+\bu +\ell \bigr) - \frac{m \bu}{(y-1)!} \prod_{\ell=2}^{y} \bigl( \bx \bu+\bu +\ell \bigr).
  \end{equation*}
  which is the expected formula \eqref{formule_ZF}. 
\end{proof}

Let $t_n$ be the linear arbor with $n$ vertices and no multiple
vertex. It follows from the general \Cref{zeta_A} at $m=1, x=n+2, y=n$
that the Zeta polynomial of the poset $P_{t_n}$ for the arbor $t_n$ is
the same as the known Zeta polynomial of the noncrossing-partitions
lattice, first computed by Edelman~\cite{edelman}, namely
\begin{equation*}
  \frac{1}{n!} \prod_{j=0}^{n-1} \bigl((n+1) u - j\bigr).
\end{equation*}

\subsection{$M$-triangle}

Here we will describe the $M$-triangles of the posets $\po^F_{m,x,y}$.

Because all intervals between the unique minimum and any element are
products of total orders, and by the same reasoning as in
\Cref{subsec:M}, the $M$-triangle is the sum
\begin{equation}
  \label{M_and_K_for_B}
  \sum_{p \in \po^F(m,x,y)} {(XY)}^{\haut(a)} {(1-1/X)}^{\nz(a)},
\end{equation}
where $\haut(a)$ is the sum of coordinates and $\nz(a)$ is the number of non-zero coordinates.

\begin{lemma}
  Assume that $y >0$ and $x > 1+my$. Then
  \begin{equation}
    \label{recurrence_M_A}
    M_{m,x,y} = M_{m,x-1,y} + (1-1/X) \sum_{j=1}^{y} {(XY)}^j M_{m,x-1,y-j}.
  \end{equation}
\end{lemma}
\begin{proof}
  Let us define a bijection $\varphi$ from $\po^F_{m,x,y}$ to the
  disjoint union
  \begin{equation*}
    \bigsqcup_{j=0}^{y} \po^F_{m,x-1,y-j}.
  \end{equation*}
  The map $\varphi$ is defined by deleting the subword
  consisting of the $a_{x-1}$ steps $\mY$ just before the last $\mX$
  step, and also removing that $\mX$ step. Let $j$ denote
  $a_{x-1}$. If $j=0$, then one obtains a path in $\po^F_{m,x-1,y}$
  because $x > 1+my$. Otherwise, $j$ can take any value between $1$
  and $y$, and then the path obtained after deletion belongs to
  $\po^F_{m,x-1,y-j}$. All this clearly defines a bijection $\varphi$,
  as one can insert back the removed subword, depending on $j$, just
  after the last $\mX$.

  If $j=0$, the bijection $\varphi$ preserves both the height $\haut$
  and $\nz$. Otherwise, the height $\haut$ is decreased by $j$ and the
  number $\nz$ is decreased by $1$. This gives the expected formula, using \eqref{M_and_K_for_B}.
\end{proof}

Together with the initial conditions
\begin{equation}
  M_{m,x,0} = 1
\end{equation}
because there is only one path in $\po^F_{m,x,0}$ and
\begin{equation}
  M_{m,my+1,y} = M_{m,my+1,y-1}
\end{equation}
because of the isomorphism of \Cref{iso_poF}, this determines everything
by induction.

\begin{proposition}
  The $M$-triangle of $\po^F_{m,x,y}$ is the sum of
  \begin{equation}
    \frac{1}{\ell!}  \binom{\ell}{k} \left( x-1-m\ell \right)\prod_{r=2}^{\ell}(x-r+k) {(-X)}^k {(-Y)}^{\ell},
  \end{equation}
  for $0 \leq k \leq \ell \leq y$.
\end{proposition}
\begin{proof}
  By induction. The statement holds in the two initial cases by
  inspection. 

  Let us expand the right-hand side of \eqref{recurrence_M_A} as the sum of two sums, with the leftmost term of \eqref{recurrence_M_A} being incorporated into the first sum:
  \begin{equation*}
    \sum_{j=0}^y \sum_{0 \leq k \leq \ell \leq y-j} \frac{1}{\ell!}\binom{\ell}{k} (x-2-m \ell) \prod_{r=2}^{\ell} (x-r-1+k) {(-X)}^{k+j} {(-Y)}^{\ell+j}
  \end{equation*}
  and
  \begin{equation*}
    \sum_{j=1}^y \sum_{0 \leq k \leq \ell \leq y-j} \frac{1}{\ell!}\binom{\ell}{k} (x-2-m \ell) \prod_{r=2}^{\ell} (x-r-1+k) {(-X)}^{k+j-1} {(-Y)}^{\ell+j}.
  \end{equation*}
  After a careful manipulation of sums, one obtains
  \begin{equation*}
    \sum_{0 \leq k \leq \ell \leq y} \sum_{j=0}^k \frac{1}{(\ell-j)!}\binom{\ell-j}{k-j}
    (x-2-m(\ell-j)) \prod_{r=2}^{\ell-j} (x-r-1+k-j){(-X)}^k {(-Y)}^{\ell}  
  \end{equation*}
  and
  \begin{equation*}
    \sum_{0 \leq k < \ell \leq y} \sum_{j=1}^{k+1} \frac{1}{(\ell-j)!}\binom{\ell-j}{k-j+1}
    (x-2-m(\ell-j)) \prod_{r=2}^{\ell-j} (x-r+k-j){(-X)}^k {(-Y)}^{\ell}.  
  \end{equation*}
  We will now concentrate on the individual coefficients of
  ${(-X)}^k {(-Y)}^{\ell}$ in these two sums.  Each coefficient is an
  affine function with respect to $m$. Let us first handle the
  constant part w.r.t. $m$. Using twice \Cref{summing_1}, one gets the two
  contributions
  \begin{equation*}
    \frac{1}{k! (\ell-k)!} \prod_{r=1}^\ell (x-1-r+k)
    \quad\text{and}\quad
    \frac{1}{k! (\ell-k-1)!} \prod_{r=1}^{\ell-1} (x-1-r+k).
  \end{equation*}
  Summing both terms, one gets as expected
  \begin{equation*}
    \frac{1}{k!(\ell-k)!} (x-1) \prod_{r=2}^{\ell}(x-r+k).  
  \end{equation*}
  Let us turn to the linear part w.r.t. $m$. Removing the factor $-m$ and using
  twice \Cref{summing_2}, one obtains two contributions
  \begin{equation*}
    \frac{1}{k!(\ell-k)!} \frac{(x-1)\ell-k}{(x-1)(x-2)} \prod_{r=1}^{\ell} (x-1-r+k)
  \end{equation*}
  and
  \begin{equation*}
    \frac{1}{k!(\ell-k-1)!} \frac{(x-1)(\ell-1)-k}{(x-1)(x-2)} \prod_{r=1}^{\ell-1} (x-1-r+k).
  \end{equation*}
  Summing both terms gives as expected
  \begin{equation*}
    \frac{\ell}{k!(\ell-k)!} \prod_{r=2}^{\ell} (x-r+k). 
  \end{equation*}
\end{proof}

The special case $m=1$ and $y=n$ and $x=n+2$, is related to the linear arbor
$t_n$ of size $n$ with no multiple vertex.
\begin{proposition}
  The $M$-triangle of the poset $P_{t_n}$ for the arbor $t_n$ is
  \begin{equation*}
    M^{\TA}_n = \sum_{0 \leq k \leq \ell \leq n} \frac{n+1-\ell}{n+k-\ell+1}\binom{n+k}{n} \binom{n}{\ell-k} {(-X)}^{k} {(-Y)}^{\ell}.  
  \end{equation*}
\end{proposition}

This formula is essentially the same as the $F$-triangle for the
cluster complex of type $\TA_n$. Recall that this $F$-triangle is
given by \cite[Prop. 11]{chap_2004}:
\begin{equation}
  F^{\TA}_n = \sum_{k=0}^{n}\sum_{\ell=0}^n \frac{\ell+1}{k+\ell+1}\binom{n}{k+\ell}\binom{n+k}{n}X^k Y^{\ell}.
\end{equation}

It follows easily that
\begin{equation}
  M^{\TA}_{n}(X,Y) = F^{\TA}_n(-X, -1/Y) \cdot {(-Y)}^n.
\end{equation}

\begin{proposition}
  The transmuted $M$-triangle of $\po^F_{1,n+2,n}$ is the $M$-triangle
  of the poset of noncrossing-partitions of type $\TA_n$.
\end{proposition}
\begin{proof}
  The transmutation of $M^{\TA}_{n}$ is
  \begin{equation*}
    M^{\TA}_{n}\left(\frac{1-Y}{1-XY},1-XY\right) = F^{\TA}_n\left(\frac{1-Y}{XY-1}, \frac{1}{XY-1} \right) \cdot {(XY-1)}^n.
  \end{equation*}
  By the proof of the $F=M$ conjecture in \cite{athanasiadis} and the transformations in \Cref{app:triangles}, the $M$-triangle
  for the poset of noncrossing-partitions of type $\TA_n$ is given by
  \begin{equation*}
    N(X,Y) = F^{\TA}_n\left(\frac{Y(X-1)}{1-XY}, \frac{XY}{1-XY}\right) \cdot {(1-XY)}^n.
  \end{equation*}
  Because the poset of noncrossing-partitions is self-dual, one has the symmetry
  \begin{equation*}
    N(X,Y) = N(1/Y,1/X) \cdot {(XY)}^n.
  \end{equation*}
  Combining these formulas gives the expected equality.
\end{proof}

An explicit formula for the $M$-triangle of the poset of
noncrossing-partitions of type $\TA_n$, not needed here, can be found
in~\cite[Prop. A]{kratt_F} with parameter $m=1$.

Recall that $t_n$ denotes the arbor of size $n$ with no multiple
vertex. It follows from the previous results that the poset $P_{t_n}$
is a transmuted partner of the poset of noncrossing partitions of type
$\TA_n$. In particular, the $h$-vector of $t_n$ is equal to the
$h$-vector of the generalized associahedra of type $\TA_n$.

\section{Arbors of type $\TB$}\label{sec:B}

Let $n \geq 1$ and $k\geq 0$. Let $\po^{\TB}_{n,k}$ be the set of
$n$-tuples $(x_1,\ldots,x_n)$ of nonnegative integers with sum less
than or equal to $k$. This is a poset for term-wise comparison.

The poset $\po^{\TB}_{n,n}$ is exactly the poset $P_t$ attached by the
general construction to the arbor with one vertex containing $n$
elements.

The posets $\po^{\TB}_{n,k}$ are graded by the height function $\haut$
which is the sum of coordinates.  They have a unique minimum element
$\bot$. The definition of the poset $\po^{\TB}_{n,k}$ implies that its
intervals are products of total orders.
 

\begin{lemma}\label{ideal_sup_TB}
  In a poset $\po^{\TB}_{n,k}$, the upper ideal generated
  by any element of height $j$ is isomorphic to $\po^{\TB}_{n,k-j}$.
\end{lemma}
\begin{proof}
  This follows directly from the definition. The bijection is given by substraction of the chosen minimal element.
\end{proof}



\begin{theorem}\label{zeta_B}
  The Zeta polynomial of $\po^{\TB}_{n,k}$ is equal to $\binom{n(u-1)+k}{k}$.
\end{theorem}

\begin{proof}
  Let us consider a chain $e_1 \leq \cdots \leq e_{u-1}$ in
  $\po^{\TB}_{n,k}$. One can describe the chain by its first
  element $e_1$ together with all differences $u_{i+1}-u_{i}$ for
  $1 \leq i \leq u-2$. This gives a bijection with $(u-1)$-tuples of $n$
  nonnegative integers whose total sum is less than or equal to
  $k$. The number of these tuples is the expected binomial coefficient.
\end{proof}

It follows from the general \Cref{zeta_B} at $k = n$ that the Zeta
polynomial of the poset $P_t$ for the arbor with just one vertex of
cardinality $n$ is the same as the known Zeta polynomial of the
noncrossing-partitions lattice of type $\TB$, first computed by Montenegro according to Reiner, see \cite{montenegro93,reiner}.

\smallskip

Let us now turn to the $M$-triangles of the posets $\po^{\TB}_{n,k}$, namely
\begin{equation*}
  M_{n,k} = \sum_{a \leq b} \mu(a,b) X^{|a|} Y^{|b|}.
\end{equation*}

As said above, by definition of $\po^{\TB}_{n,k}$, all its intervals
are products of total orders. By the same reasoning as in \Cref{subsec:M},
this leads to the expression
\begin{equation}
  M_{n,k} = \sum_{a \in \po^{\TB}_{n,k}} {(XY)}^{\haut(a)} {(1-1/X)}^{\nz(a)},
\end{equation}
where $\nz$ is the number of non-zero coordinates.

In this sum, let us write $a$ as a pair $(j,b)$ where $j$ is the first
coordinate and $b$ are the remaining coordinates seen as an element in
$\po^{\TB}_{n-1,k-j}$.

By separating the contribution for $j=0$, this allows to rewrite the sum as
\begin{equation}
  \label{recurrence_M_B}
  M_{n,k} = M_{n-1,k} + (1-1/X) \sum_{j=1}^k {(XY)}^j M_{n-1,k-j}.
\end{equation}
Together with the initial condition
\begin{equation}
  \label{init_M_B}
  M_{0, k} = 1
\end{equation}
for the singleton poset, this equation determines the $M$-triangle by
induction on $n$.




\begin{proposition}
  The $M$-triangle for the poset $\po^{\TB}_{n,k}$ is
  \begin{equation}
    \sum_{r=0}^{k} \sum_{\ell=0}^{k-r} \binom{\ell+n-1}{n-1} \binom{n}{r} {(-X)}^{\ell} {(-Y)}^{\ell+r}.
  \end{equation}
\end{proposition}
\begin{proof}
  This is done by checking that this formula satisfies the induction
  \eqref{recurrence_M_B} and the initial condition \eqref{init_M_B}. The initial condition is clear.
  For the induction step, the right hand side of \eqref{recurrence_M_B} is the sum of
  \begin{equation*}
    \sum_{j=0}^k \sum_{r=0}^{k-j}\sum_{\ell=0}^{k-j-r} \binom{\ell+n-2}{n-2}\binom{n-1}{r}{(-X)}^{\ell+j} {(-Y)}^{\ell+r+j}
  \end{equation*}
  and
  \begin{equation*}
    \sum_{j=1}^k \sum_{r=0}^{k-j}\sum_{\ell=0}^{k-j-r} \binom{\ell+n-2}{n-2}\binom{n-1}{r}{(-X)}^{\ell+j-1} {(-Y)}^{\ell+r+j}.
  \end{equation*}
  After replacing the
  summation variable $j$ by a new summation variable
  $\ell' = \ell +j$ and some further careful manipulations of summations, they become
  \begin{equation*}
    \sum_{r=0}^{k} \sum_{\ell'=0}^{k-r} \binom{\ell'+n-1}{n-1}\binom{n-1}{r} {(-X)}^{\ell'} {(-Y)}^{\ell'+r}
  \end{equation*}
  and
  \begin{equation*}
    \sum_{r=1}^{k} \sum_{\ell'=0}^{k-r} \binom{\ell'+n-1}{n-1}\binom{n-1}{r-1} {(-X)}^{\ell'} {(-Y)}^{\ell'+r}.
  \end{equation*}
  Adding these two expressions then gives the expected expression.
\end{proof}

The special case $k=n$ is related to the arbor $t_n$ with one vertex
containing $n$ elements.
\begin{proposition}
  The $M$-triangle of the poset $P_{t_n}$ for the arbor $t_n$ is
  \begin{equation*}
    M_{n,n}(X,Y) = \sum_{r=0}^{n} \sum_{\ell=0}^{n-r} \binom{\ell+n-1}{n-1} \binom{n}{r} {(-X)}^{\ell} {(-Y)}^{\ell+r}.
  \end{equation*}
\end{proposition}

This formula is essentially the same as the $F$-triangle for the
cluster complex of type $\TB_n$. Recall that this $F$-triangle is
given by \cite[Prop. 13]{chap_2004}:
\begin{equation}
  F^{\TB}_n = \sum_{k=0}^{n}\sum_{\ell=0}^n \binom{n}{k+\ell}\binom{n+k-1}{n-1}X^k Y^{\ell}.
\end{equation}

It follows easily that
\begin{equation}
  M_{n,n}(X,Y) = F^{\TB}_n(-X, -1/Y) \cdot {(-Y)}^n.
\end{equation}

\begin{proposition}
  The transmuted $M$-triangle of $\po^{\TB}_{n,n}$ is the $M$-triangle
  of the poset of noncrossing-partitions of type $\TB_n$.
\end{proposition}
\begin{proof}
  The transmutation of $M_{n,n}$ is
  \begin{equation*}
    M_{n,n}\left(\frac{1-Y}{1-XY},1-XY\right) = F^{\TB}_n\left(\frac{1-Y}{XY-1}, \frac{1}{XY-1} \right) \cdot {(XY-1)}^n.
  \end{equation*}
  By the proof of the $F=M$ conjecture in~\cite{athanasiadis} and the
  conversion formulas in \Cref{app:triangles}, the $M$-triangle for the
  poset of noncrossing-partitions of type $\TB_n$ is given by
  \begin{equation*}
    N(X,Y) = F^{\TB}_n\left(\frac{Y(X-1)}{1-XY}, \frac{XY}{1-XY}\right) \cdot {(1-XY)}^n.
  \end{equation*}
  Because the poset of noncrossing-partitions is self-dual, one has the symmetry
  \begin{equation*}
    N(X,Y) = N(1/Y,1/X) \cdot {(XY)}^n.
  \end{equation*}
  Combining these formulas gives the expected equality.
\end{proof}

An explicit formula for the $M$-triangle of the poset of
noncrossing-partitions of type $\TB_n$, not needed here, can
be found in~\cite[Prop. B]{kratt_F} with parameter $m=1$.

It follows that the poset $P_{t_n}$ for the arbor $t_n$ with one vertex
containing $n$ elements is a transmuted partner of the poset of
noncrossing partitions of type $\TB_n$. In particular, the $h$-vector
of $t_n$ is equal to the $h$-vector of the generalized associahedra of
type $\TB_n$.



Let us also record here the following value, useful in the next section.

\begin{proposition}\label{valeur_K_B}
  For $n \geq 1$, let $t_n$ be the arbor with just one vertex with $n$ elements. Then
  \begin{equation}
    K_{t_n}(X,Y) = \sum_{0 \leq j \leq k \leq n} \binom{n}{j} \binom{k-1}{k-j} X^j Y^k.
  \end{equation}
\end{proposition}
\begin{proof}
  This is a direct application of \Cref{recurrence_K} in the case
  without sub-tree. The double sum there is restricted to
  the unique term $\ell = j$ and $m=k$.
\end{proof}

\section{Arbors and halohedra}\label{sec:halohedra}

The halohedra form a sequence of simple convex polytopes. They have a
topological interpretation as a compactification of a configuration
space, similar to the case of associahedra and cyclohedra
\cite{devadoss1, devadoss2}. They also have a combinatorial
description in terms of ``design tubings'' of the cycle graphs, that
makes them close relatives of graph-associahedra. The halohedra have
also appeared in mathematical physics \cite{salvatori,salvatori2}.

There is one halohedron $H_n$ in dimension $n$ for each dimension
$n \geq 1$, starting with a segment and a pentagon. The number of
vertices of $H_n$ is
\begin{equation*}
  \label{cardinal_halo}
  \frac{3 n - 1}{n} \binom{2n-2}{n-1}.
\end{equation*}

For $n \geq 2$, let $h_n$ be the arbor with $n-1$ elements in the root
vertex plus one leaf with one element.
\begin{proposition}
  For $n \geq 2$, the cardinality of $P_{h_n}$ is the number of
  vertices of $H_n$.
\end{proposition}
\begin{proof}
  This is done using \Cref{calcul_ehr} at $m=1$, using as starting
  point the generating polynomial $1+X$ for the leaf vertex. One finds
  that the generating polynomial for $P_t$ with respect to height is
  \begin{equation*}
    \sum_{j=0}^{n} \left(\binom{n+j-2}{n-2}+\binom{n+j-3}{n-2} \right) X^j.
  \end{equation*}
  This can be evaluated at $X=1$ to 
  \begin{equation*}
    \binom{2n-1}{n-1} + \binom{2n-2}{n-1},
  \end{equation*}
  which is the expected result.
\end{proof}

For $n \geq 1$, let $h'_n$ be the arbor with one element in the root
vertex plus one leaf with $n-1$ elements. This is the reverse arbor of $h_n$.
\begin{proposition}
  For $n \geq 1$, the cardinality of $P_{h'_n}$ is the number of
  vertices of $H_n$.
\end{proposition}
\begin{proof}
  This is done using \Cref{calcul_ehr} at $m=1$.
  For the leaf vertex with $n-1$ elements, the generating polynomial is
  \begin{equation*}
    \sum_{j=0}^{n-1}  \binom{n-2+j}{n-2} X^j.
  \end{equation*}
  One finds that the
  generating polynomial for $P_t$ with respect to height is
  \begin{equation*}
    \sum_{j=0}^{n-1} \binom{n-1+j}{n-1} X^j + \binom{2n-2}{n-1} X^n.
  \end{equation*}
  At $X=1$, this simplifies to
  \begin{equation*}
    \binom{2n-1}{n}+\binom{2n-2}{n-1},
  \end{equation*}
  which is the expected result.
\end{proof}

Let us now compute the $h$-vectors for these two families of arbors.

\begin{proposition}\label{hvecteur_halo}
  As a polynomial, for $n \geq 2$, the $h$-vector of $h_n$ is
  \begin{equation}
    \sum_{j=0}^{n-1} \binom{n-1}{j}\binom{n}{j} X^j +
    \sum_{j=1}^{n} \binom{n-1}{j-1}^2 X^j.
  \end{equation}
\end{proposition}
\begin{proof}
  One computes $K_{h_n}(X,1)$ using \Cref{recurrence_K1}, with
  starting point the polynomial $1+XY$ for the leaf vertex. One gets
  exactly the expected formula for the coefficient of $X^j$.
\end{proof}

One can easily check that this polynomial is palindromic.

\begin{proposition}
  As a polynomial, for $n \geq 1$, the $h$-vector of $h'_n$ is
  \begin{equation}
    \sum_{j=0}^{n-1} \binom{n-1}{j}^2 X^j +
    \sum_{j=1}^{n} \binom{n-1}{j-1}\binom{n}{j} X^j.
  \end{equation}
\end{proposition}
\begin{proof}
  One computes $K_{h_n}(X,1)$ using \Cref{recurrence_K1}, with
  starting point the polynomial from \Cref{valeur_K_B}.
  One gets two terms for the coefficient of $X^j$. The first one is
  \begin{equation*}
    \sum_{k=j}^{n-1} \binom{n-1}{j}\binom{k-1}{k-j}
  \end{equation*}
  for $0 \leq j \leq n-1$, which sums to $\binom{n-1}{j}^2$. The second term is
  \begin{equation*}
    \sum_{k=j-1}^{n-1} (n-k) \binom{n-1}{j-1}\binom{k-1}{k-j+1}
  \end{equation*}
  for $1 \leq j \leq n$, which sums to $\binom{n-1}{j-1}\binom{n}{j}$.
\end{proof}

It follows that the $h$-vectors for $h_n$ and
$h'_n$ coincide, as predicted by \Cref{EZ_conjecture2}.

The $h$-vectors for the halohedra are known by \cite[Theorem
6.1.11]{almeter}, which says, once slightly adapted, that the
generating series of $h$-vectors of halohedra for $n \geq 1$ is
\begin{equation*}
  h = \frac{1+(1+X)s}{2\sqrt{1-2(1+X)s+{(X-1)}^2s^2}} - \frac{1}{2} = (X+1)s + (X^2 + 3 X + 1) s^2 + \cdots
\end{equation*}
as a formal power series in $s$. This series therefore satisfies the algebraic equation
\begin{equation*}
  (1-2(1+X)s+{(X-1)}^2s^2) {(2 h+1)}^2 = {(1+(1+X)s)}^2,
\end{equation*}
which can also be written as
\begin{equation}
  \label{eq_alg_halo}
  (1-2(1+X)s+{(X-1)}^2s^2) (h + h^2) = X s^2 + X s + s.
\end{equation}

It is proved in \Cref{calcul_halo} that the generating series of the
common $h$-vectors for the two families of arbors considered above
satisfies the algebraic equation \eqref{eq_alg_halo}. As they start
with the same initial terms, this implies that they coincide with the
$h$-vectors of the halohedra.

Let us note that halohedra are related to the Coxeter type $\TD_n$:
for $n\geq 2$, the $h$-vector for $H_n$ is equal, as a polynomial in $X$, to the sum of the
$h$-vector for the associahedra of type $\TD_n$, with $X$ times the
$h$-vector for the associahedra of type $\TA_{n-2}$ (see
\cite[Fig. 5.12]{Fomin_Reading} for the necessary formulas).




\section{Hochschild polytopes}\label{sec:hochschild}

The Hochschild polytopes\footnote{They are also known as
  \textit{freehedra} in anglo-greek.} were introduced by Saneblidze in the context of
algebraic topology of loop spaces \cite{saneblidze, san_rivera}. They
form a sequence of simple polytopes, one in each dimension $n \geq 1$,
starting with a segment in dimension $1$ and a pentagon in dimension
$2$. The Hochschild polytope in dimension $n$ has
$2^{n - 2} \times (n + 3)$ vertices and its $h$-vector is
${(X+1)}^{n-2}(X^2+(n+1)X+1)$ \cite{combe_hoch}.

In the study of greedy Tamari lattices in \cite{dexter}, a lattice
structure has been introduced on the vertices of each Hochschild
polytope, in such a way that cover relations are edges of the
polytope. These Hochschild lattices have been studied in detail in
\cite{combe_hoch}, who showed that they are congruence-uniform. Some
further work can be found in \cite{poliakova2020,
  poliakova2021,muhle,pilaud_polia}.

\smallskip

Let us consider the following sequence of arbors $t_n$ for $n \geq
1$. The arbor $t_n$ is made of a root vertex with one element, on
which are grafted $n-1$ vertices, all with one element. Here is the
example of $t_4$:
\begin{center}
\begin{tikzpicture}
[grow=down]
\tikzstyle{every node}=[draw,shape=circle]

\node [fill=racines,thick] {$a$}
  child {node {$b$}}
  child {node {$c$}}
  child {node {$d$}}
;
\end{tikzpicture}.
\end{center}

The Hochschild polytope in dimension $n$ is
related to the poset $P_{t_n}$ by the following equalities.
\begin{proposition}\label{conjecture-hoch}
  The number of vertices in the Hochschild polytope in dimension $n$
  is the number of elements of the poset $P_{t_n}$. The $h$-vector of
  the Hochschild polytope is equal to the $h$-vector of $t_n$.
\end{proposition}
\begin{proof}
  Using the recursive description of the Ehrhart polynomial at $m=1$
  in \Cref{calcul_ehr} and the initial value $F_{\circ,1}=1+X$ for the
  arbor with one vertex, one finds that the number of elements of
  $P_{t_n}$ is
  \begin{equation*}
    \sum_{\ell=0}^{n-1} (n+1-\ell)\binom{n-1}{\ell}.
  \end{equation*}
  By breaking this sum into two, one gets
  $ (n+1) 2^{n-1} - (n-1) 2^{n-2}$. For the finer statement about the
  $h$-vectors, one uses the recursive description in
  \Cref{recurrence_K1} and the initial value $K_{\circ}=1+XY$ for the arbor with
  one vertex. One first gets that
  \begin{equation*}
    K_{t_n}(X,1) = \sum_{j=0}^{n-1} \binom{n-1}{j}X^j + \sum_{1 \leq j \leq n} (n-j+1) \binom{n-1}{j-1} X^j.
  \end{equation*}
  These two simple sums can be computed as above and the result
  follows. 
\end{proof}

In fact, there should be a more precise relationship. Let us consider
the spherical simplicial complex $SC_{n}$ which is the polar dual of the
Hochschild polytope in dimension $n$.
\begin{conjecture}
  In the simplicial complex $SC_{n}$, there exists a facet $C$
  such that the $F$-triangle for the pair $(SC_{n}, C)$ coincides with
  the $F$-triangle associated with the transmuted $M$-triangle of
  $P_{t_n}$.
\end{conjecture}
This can been checked concretely in small cases. It has been checked
by computer for $n \leq 9$. For $n \geq 3$, there seems to be exactly
$4$ facets with the expected property.

Moreover, it seems that the core-label-order of the Hochschild lattice
of dimension $n$ could be a transmuted partner of the poset
$P_{t_n}$. The coincidence of the Zeta polynomial with that of $P_{t_n}$ and the
transmutation property for $M$-triangles have been checked by computer
up to $n=9$.


\smallskip

Using \verb|FriCAS| \cite{fricas}, one can guess the following
conjectural formulas for various generating series in $s$ for the arbors $t_n$.
For the Zeta polynomials of the posets $P_{t_n}$:
\begin{equation*}
  1 + \sum_{n \geq 1} \mathsf{Z}_{t_n}(u, 1) s^n \stackrel{?}{=}
  \exp\left( \int \frac{u}{(1-us)(1+s-us)} ds \right). 
\end{equation*}
For the $M$-triangles of the posets $P_{t_n}$:
\begin{equation*}
  1 + \sum_{n \geq 1} M_{t_n}(X,Y) s^n \stackrel{?}{=} \frac{{\left(X Y s - Y s - 1\right)} {\left(X Y s - 1\right)}}{{\left(2 \, X Y s - Y s - 1\right)} {\left(X Y s - Y s + s - 1\right)}}. 
\end{equation*}
For the Ehrhart polynomials of the polytopes $Q_{t_n}$:
\begin{equation*}
  1 + \sum_{n \geq 1} \operatorname{Ehr}_{t_n}(u) s^n \stackrel{?}{=} 
\frac{1}{2} \left( 1 -\frac{s}{{\left(u s + s - 1\right)}} -\frac{s - 1}{{\left(u s + s - 1\right)}^{2}}\right). 
\end{equation*}
For the Laplace transform of volume functions of the polytopes $Q_{t_n}$:
\begin{equation*}
  \sum_{n \geq 1} L_{t_n}(E,V) s^n \stackrel{?}{=}  V \left(\frac{s}{E V s - V s + 1} + \frac{E s}{E s - 1}\right). 
\end{equation*}


\appendix

\section{Conversion formulas}\label{app:triangles}

We gather here the conversion formulas that we use between three
different kinds of polynomials in two variables, namely $F$-triangles,
$M$-triangles and $H$-triangles. These triangles and rational
change-of-variables between them were introduced in \cite{chap_2004,
  chap_full} in the context of cluster fans and generalized
associahedra. They are called \textit{triangles} because of the shape of their
Newton polygons. Each triangle comes with a size parameter $n$, which
is the maximal value of either coordinate in the Newton polygon.

Given a pure simplicial complex of spherical topology $SC$ and a
chosen facet $C$, one can define the $F$-triangle as the
following sum over all faces:
\begin{equation}
  \label{defiF}
  F(X,Y) = \sum_{s \in SC} X^{\posi(s)} Y^{\nega(s)},
\end{equation}
where $\nega(s)$ (resp. $\posi(s)$) is the number of elements in $s$
that belong (resp.\ do not belong) to $C$.

Given a graded poset $P$, one can define the $M$-triangle of $P$ as
the following sum of Möbius numbers weighted using the rank function
$\haut$:
\begin{equation}
  \label{defiM}
  M(X, Y) = \sum_{a \leq b} \mu(a,b)  X^{\haut(a)} Y^{\haut(b)}.
\end{equation}

The $M$-triangles and $F$-triangles are converted to each other by
\begin{equation}
  \label{F_from_M}
  F(X,Y) = M\left(\frac{Y}{Y-X}, \frac{Y-X}{1+Y}\right) \cdot {(1+Y)}^n,
\end{equation}
where $n$ is the size parameter, and the similar birational formula
\begin{equation}
  \label{M_from_F}
  M(X,Y) = F\left(\frac{Y(X-1)}{1-XY},\frac{XY}{1-XY}\right)\cdot {(1-XY)}^n
\end{equation}
in the other direction.

We will not define here the $H$-triangles, that are related to the
cohomology ring of the normal fan of the generalized associahedra. We
give the following formulas because they motivate the definition of
transmutation, as explained at the end of \Cref{sec:transmutation}.

The $M$-triangles and $H$-triangles are converted to each other by
\begin{equation}
  \label{M_from_H}
  m(X,Y) = h\bigl((X - 1) Y / (1 - Y), X / (X - 1)\bigr) \cdot {(1 - Y)}^n
\end{equation}
where $n$ is the size parameter, and the similar birational formula
\begin{equation}
  \label{H_from_M}
  h(X,Y) = m\bigl(Y / (Y - 1), (Y -  1) X / (1 + (Y - 1) X)\bigr) \cdot {(1 + (Y - 1) X)}^n
\end{equation}
for the conversion in the other way.

\section{Two summation lemmas}

These lemmas are used in \Cref{sec:A}.

\begin{lemma}\label{summing_1}
  One has
  \begin{equation}
    \sum_{j \geq 0} \frac{1}{(k-j)!} \prod_{r=2}^{\ell-j} (x-r+k-j) = \frac{1}{k!} \frac{1}{x-1} \prod_{r=1}^{\ell} (x-r+k).
  \end{equation}
\end{lemma}
\begin{proof}
  Write the left hand side as the value at $j=0$, namely
  \begin{equation*}
    \frac{1}{k!} \prod_{r=2}^{\ell} (x-r+k),
  \end{equation*}
  times the hypergeometric series
  \begin{equation*}
    \pFq{2}{1}{-k,1}{-k+2-x}.
  \end{equation*}
  The later evaluates to $\frac{x+k-1}{x-1}$ by Chu-Vandermonde.
\end{proof}


\begin{lemma}\label{summing_2}
  One has
  \begin{equation}
    \sum_{j \geq 0} \frac{1}{(k-j)!} (\ell-j) \prod_{r=2}^{\ell-j} (x-r+k-j) = \frac{1}{k!} \frac{x\ell-k}{x (x-1)} \prod_{r=1}^{\ell} (x-r+k).
  \end{equation}
\end{lemma}
\begin{proof}
  Write the left hand side as the value at $j=0$, namely
  \begin{equation*}
    \frac{1}{k!} \ell \prod_{r=2}^{\ell} (x-r+k),
  \end{equation*}
  times the hypergeometric series
  \begin{equation*}
    \pFq{3}{2}{-k,1,1-\ell}{-k+2-x,-\ell}.
  \end{equation*}
  This hypergeometric series can be transformed\footnote{We trust SageMath or Maxima here. This should follow from some contiguity relation.} into
  \begin{equation*}
    \pFq{2}{1}{-k,1}{-k+2-x} -\frac{k}{\ell (k+x-2)} \pFq{2}{1}{1-k,2}{-k-x+3}.
  \end{equation*}
  The later evaluates to $\frac{(-k+\ell x)(x+k-1)}{\ell x(x-1)}$ by Chu-Vandermonde.
\end{proof}



\section{Some computations for halohedra}\label{calcul_halo}

Following \Cref{hvecteur_halo}, let us define two power series in the
variable $s$:
\begin{equation}
  \label{defiAB}
  A = \sum_{n\geq 1} \sum_{j=1}^{n} \binom{n-1}{j-1}^2 X^j s^n\quad\text{and}\quad
  B = \sum_{n\geq 1} \sum_{j=0}^{n-1} \binom{n-1}{j} \binom{n}{j} X^j s^n.
\end{equation}

The aim of this section is to prove that $A+B$ coincide with the known
generating function for $h$-vectors of halohedra, characterized by
\eqref{eq_alg_halo} and its first few terms. This is done with the
help of the software \verb|FriCAS| \cite{fricas} and the
\verb|ore_algebra| package for \verb|sagemath| \cite{ore_algebra,
  sagemath}.

Let us write $P$ for the polynomial ${(X-1)}^2 s^2 - 2 (X+1)s +1$.

First, one can guess an algebraic equation satisfied by $A$:  
\begin{equation}
  \label{equation_A}
  P A^2 = X^2 s^2.
\end{equation}
The operator $s \partial_s$ acts on both sides by multiplication by
$2$. After division by $2 A$, one deduces a possible differential
equation for $A$:  
\begin{equation*}
  s P \partial_s A + (X s + s - 1) A = 0,
\end{equation*}
with initial condition $A_0 = 0$. This is equivalent to the following
possible recurrence for the coefficients of $A$:  
\begin{equation*}
  {(X-1)}^2 n A_n - (X+1) (2n+1) A_{n+1} + (n+1) A_{n+2} = 0,
\end{equation*}
with initial conditions $A_0 =0$ and $A_1 = X$.  One then checks  
that the formula~\eqref{defiAB} for $A_n$ indeed satisfies this
recurrence with the same initial conditions. So $A$ is the unique
power series solution of~\eqref{equation_A} with zero constant term.

Let us now turn to the case of $B$. First, one can guess an algebraic
equation satisfied by $B$: 
\begin{equation}
  \label{equation_B}
  P (B + B^2) = s.
\end{equation}
One can also guess a differential equation for $B$:  
\begin{equation*}
  X s - s + 1 + 2 (X s-s + 1) B + (X s-s-1) P \partial_s B = 0
\end{equation*}
with initial condition $B_0 = 0$. Let us consider in the ring
$\QQ[X,s,B,\partial_{s}B]$ the ideal generated by~\eqref{equation_B},
its formal derivative with respect to $s$ and this differential
equation. Computing the elimination ideal with respect to
$\partial_s B$ gives back the principal ideal generated by~\eqref{equation_B}. This proves that the differential equation is
compatible with~\eqref{equation_B}. This differential equation is
equivalent to the following possible recurrence for the coefficients
of $B$: 
\begin{multline*}
  {(X-1)}^3 n  B_n -(X-1) (3 n X+n+3 X-1) B_{n+1} \\+(3 n X+n+6 X+4) B_{n+2} -(n+3) B_{n+3} = 0,
\end{multline*}
with initial conditions $B_0 = 0$, $B_1 = 1$ and $B_2 = 1+2X$.

One then checks that the formula~\eqref{defiAB} for $B_n$ indeed 
satisfies this recurrence with the same initial conditions. So one can
conclude that $B$ is the unique power series solution of~\eqref{equation_B} with the given initial terms.

From the algebraic equations~\eqref{equation_A} and~\eqref{equation_B}, one gets similar equations for $A/s$ and $B/s$.
One then obtains by elimination an algebraic
equation for the sum $(A+B)/s$, hence 
\begin{equation}
  P \cdot ((A+B) + {(A+B)}^2) = X s^2 + (X + 1) s.
\end{equation}
Note that this requires choosing the correct factor among two in the
reducible polynomial for $(A+B)/s$ obtained by elimination. This is
done by identifying the constant term.

This implies the expected equality between $A+B$ and the generating
series of $h$-vectors of halohedra defined by~\eqref{eq_alg_halo},
after checking enough initial terms.

\bibliographystyle{alpha} \bibliography{arbor_plain}

\end{document}